\documentclass[a4paper,11pt]{article}

\usepackage{graphicx}
\usepackage{amsmath}
\usepackage{amsfonts}
\usepackage{titlesec}
\usepackage{amsthm}
\usepackage{amssymb}
\usepackage{geometry}
\usepackage{algorithm2e}

\newcommand{\R}{\mathbb{R}}
\newcommand{\T}{\mathcal{T}}

\newcommand{\mc}[1]{\mathcal{#1}}

\newcommand{\gmin}{\gamma_{\operatorname{min}}}
\newcommand{\gmax}{\gamma_{\operatorname{max}}}

\newcommand{\uorth}{V_{u}^{\perp}}

\newcommand{\Vf}{V^{\operatorname*{f}}_{H,h}}
\newcommand{\Vfzero}{V^{\operatorname*{f}}_{H,0}}
\newcommand{\f}{{\operatorname*{f}}}
\newcommand{\cs}{{\operatorname*{c}}}
\newcommand{\Pf}{P^{\operatorname*{f}}}
\newcommand{\Pc}{P^{\operatorname*{c}}}
\newcommand{\Vc}{V^{\operatorname*{c}}_{H,h}}
\newcommand{\Vczero}{V^{\operatorname*{c}}_{H,0}}

\newcommand{\support}{\operatorname*{supp}}
\newcommand{\dimension}{\operatorname*{dim}}
\newcommand{\dia}{\operatorname*{diam}}

\newcommand{\ddiv}{\operatorname*{div}}
\newcommand{\tnorm}[1]{\|#1\|_{H^1(\Omega)}}
\newcommand{\dx}{\hspace{2pt}dx}
\newtheorem{theorem}{Theorem}[section]
\newtheorem{corollary}[theorem]{Corollary}

\newtheorem{lemma}[theorem]{Lemma}
\newtheorem{proposition}[theorem]{Proposition}

\newtheorem{problem}[theorem]{Problem}
\theoremstyle{definition}

\newtheorem{remark}[theorem]{Remark}
\begin{document}

\begin{center}
{\LARGE Two-Level discretization techniques for ground state computations of Bose-Einstein condensates}\\[2em]
\end{center}
\renewcommand{\thefootnote}{\fnsymbol{footnote}}
\renewcommand{\thefootnote}{\arabic{footnote}}

\begin{center}
{\large Patrick Henning\footnote[1]{ANMC, Section de Math\'{e}matiques, \'{E}cole polytechnique f\'{e}d\'{e}rale de Lausanne, 1015 Lausanne, Switzerland}$\hspace{0pt}^{\hspace{1pt}\ast}$,
 Axel M\r{a}lqvist\footnote[2]{Department of Mathematical Sciences, Chalmers University of Technology and University of Gothenburg, SE-41296 Gothenburg, Sweden}\renewcommand{\thefootnote}{\fnsymbol{footnote}}\setcounter{footnote}{0}
 \hspace{-3pt}\footnote{P. Henning and A. M\r{a}lqvist were supported by The G\"{o}ran Gustafsson Foundation and The Swedish Research Council.},
 \renewcommand{\thefootnote}{\arabic{footnote}}\setcounter{footnote}{2}
 Daniel Peterseim\footnote[3]{Institut f\"ur Numerische Simulation der Universit\"at Bonn, Wegelerstr. 66, 53123 Bonn, Germany}\renewcommand{\thefootnote}{\fnsymbol{footnote}}\setcounter{footnote}{1}
 \hspace{-3pt}\footnote{D. Peterseim was partly supported by the Humboldt-Universit\"at and the DFG Research Center Matheon Berlin.}}\\[2em]
\end{center}

\begin{center}
{\large{\today}}
\end{center}

\begin{center}
\end{center}

\begin{abstract}
This work presents a new methodology for computing ground states of Bose-Einstein condensates based on finite element discretizations on two different scales of numerical resolution. In a pre-processing step, a low-dimensional (coarse) generalized finite element space is constructed. It is based on a local orthogonal decomposition of the solution space and exhibits high approximation properties. The non-linear eigenvalue problem that characterizes the ground state is solved by some suitable iterative solver exclusively in this low-dimensional space, without significant loss of accuracy when compared with the solution of the full fine scale problem. The pre-processing step is independent of the types and numbers of bosons. A post-processing step further improves the accuracy of the method. We present rigorous a priori error estimates that predict convergence rates $H^3$ for the ground state eigenfunction and $H^4$ for the corresponding eigenvalue without pre-asymptotic effects; $H$ being the coarse scale discretization parameter. 
Numerical experiments indicate that these high rates may still be pessimistic. 
\end{abstract}



\section{Introduction}
\label{section-introduction}

Bose-Einstein condensates (BEC) are formed when a dilute gas of trapped bosons (of the same species) is cooled down to ultra-low temperatures close to absolute zero \cite{Bos1,RevModPhys.71.463,Einstein-24,
Pitaevskii:Stringari:2003}. In this case, nearly all bosons are in the same quantum mechanical state, which means that they loose their identity and become indistinguishable from each other. The BEC therefore behaves like one 'super particle' where the quantum state can be described by a single collective wave function $\Psi$. The dynamics of a BEC can be modeled by the time-dependent Gross-Pitaevskii equation (GPE) \cite{Gross:343403,Lieb:Seiringer:Yngvason:2000,Pitaevskii1961}, which is a nonlinear Schr\"odinger equation given by
\begin{align}
\label{time-dependent-gpe}i \hbar \hspace{2pt} \partial_t \Psi = -\frac{\hbar^2}{2m} \triangle \Psi + V_e \Psi + \frac{4 \pi \hbar^2 a N}{m} |\Psi|^2 \Psi.
\end{align}
Here, $m$ denotes the atomic mass of a single boson, $N$ the number of bosons (typically in the span between $10^3$ and $10^7$), $\hbar$ is the reduced Plank's constant and $V_e$ is an external trapping potential that confines the system. The nonlinear term in the equation describes the effective two-body interaction between the particles. If the scattering length $a$ is positive, the interaction is repulsive, if it is negative the interaction is attractive. For $a=0$ there is no interaction and \eqref{time-dependent-gpe} becomes the Schr\"odinger equation. The parameter $a$ changes according to the considered species of bosons. We only consider the case $a\ge0$ in this paper.
We are mainly interested in the ground state solution of the problem. This stationary state of the BEC is of practical relevance, e.g., in the context of atom lasers \cite{76421,Ketterle1997,PhysRevA.57.2030}. The ansatz $\Psi(x,t)= \hat{c} e^{-i\lambda \hat{t}} u(\hat{x})$, with the unknown chemical potential of the condensate $\lambda$ and a proper nondimensionalization $(x,t)\mapsto(\hat{x},\hat{t})$, reduces \eqref{time-dependent-gpe} to the time-independent GPE
\begin{align*}
-\frac{1}{2} \triangle u + V u + \beta |u|^2 u = \lambda u \qquad \mbox{with} \enspace \beta=\frac{4 \pi a N}{x_s},
\end{align*}
where $x_s$ denotes the dimensionless length unit and where $V$ denotes the accordingly rescaled potential (see, e.g., \cite{Bao:Tang:2003} for a derivation of the time-independent GPE). The ground state of the BEC is the lowest energy state of the system and is therefore stable. It minimizes the corresponding energy $$ E(v)=\int_{\mathbb{R}^d} \frac{1}{2} |\nabla v|^2 + V |v|^2 + \frac{\beta}{2} |v|^4 \dx$$
amongst all $L^2$-normalized $H^1$ functions. For any $L^2$-normalized minimizer $u$, $\lambda=E(u)+\frac{\beta}{2}\|u\|^4_{L^4(\mathbb{R}^d)}$ is the smallest eigenvalue of the GPE. In this paper, we shall focus on the computation of this ground state eigenvalue. Eigenfunctions whose energies are larger than the minimum energy are called excited states of the BEC and are not stable in general but may satisfy relaxed concepts of stability such as metastability (see \cite{PhysRevLett.82.2228}).
Numerical approaches for the computation of ground states of a BEC typically involve an iterative algorithm that starts with a given initial value and diminishes the energy of the density functional $E$ in each iteration step. Different methodologies are possible: methods related to normalized gradient-flows \cite{Bao:Chern:Lim:2006,Aftalion:Du:2001,Aftalion:Danaila:2003,Aftalion-Danaila-2004,Bao:Chern:Lim:2006,Bao:Shen:2008,Garcia-Ripoll:Perez-Garcia:2001,Bao:Du:2004,Bao:Wang:Markowich:2005,Danaila:Kazemi:2010}, methods based on a direct minimization of the energy functional \cite{Bao:Tang:2003,Caliari:et-al:2009}, explicit imaginary-time marching \cite{Chiofalo:Succi:Tosi:2000}, the DIIS method (direct inversion in the iterated subspace) \cite{PhysRevA.59.2232,Cerimele200082}, or the Optimal Damping Algorithm \cite{Cances:LeBris:2000,Cances:2000}.
We emphasize that, in any case, the dimensionality of the underlying space discretization is the crucial factor for computational complexity because it determines the cost per iteration step. The aim of this paper is to present a low-dimensional space discretization that reduces the cost per step and, hence, speeds up the iterative solution procedure considerably. 
In the literature, there are only a few contributions on rigorous numerical analysis of space discretizations of the GPE. In particular, explicit orders of convergence are widely missing. In \cite{Zhou:2004,Chen:Gong:Zhou:2010}, Zhou and coworkers proved the convergence of general finite dimensional approximations that were obtained by minimizing the energy density $E$ in a finite dimensional subspace of $H^1_0(\Omega)$. This justifies, e.g., the direct minimization approach proposed in \cite{Bao:Tang:2003}. The iteration scheme is not specified and not part of the analysis. The results of Zhou were generalized by Canc\`es, Chakir and Maday \cite{Cances:Chakir:Maday:2010} allowing explicit convergence rates for finite element approximations and Fourier expansions. A-priori error estimates for a conservative Crank-Nicolson finite difference (CNFD) method and a semi-implicit finite difference (SIFD) method were derived by Bao and Cai \cite{Bai:Cai:2013}.

In this work, we propose a new space discretization strategy that involves a pre-processing step and a post-processing step in standard $P1$ finite element spaces. The pre-processing step is based on the numerical upscaling procedure suggested by two of the authors \cite{Malqvist:Peterseim:2012} for linear eigenvalue problems. In this step, a low-dimensional approximation space is assembled. The assembling is based on some local orthogonal decomposition that incorporates problem-specific information. The constructed space exhibits high approximation properties. The non-linear problem is then solved in this low-dimensional space by some standard iterative scheme (e.g., the ODA \cite{Cances:LeBris:2000}) with very low cost per iteration step. 
The post-processing step is based on the two-grid method suggest by Xu and Zhou \cite{Xu:Zhou:2001}. We emphasize that both, pre- and post-processing, involve only the solution of linear elliptic Poisson-type problems using standard finite elements. 
We give a rigorous error analysis for our strategy to show that we can achieve convergence orders of $H^4$ for the computed eigenvalue approximations without any pre-asymptotic effects. We do not focus on the iterative scheme that is used for solving the discrete minimization problem. The various choices previously mentioned, e.g., the ODA \cite{Cances:LeBris:2000} are possible. 
Our new strategy is particularly beneficial in experimental setups with different types of bosons, because the results of the pre-processing step can be reused over and over again independent of $\beta$. Similarly, the data gained by pre-processing can be recycled for the computation of excited states. Other applications include setups with potentials that oscillate at a very high frequency (e.g., to investigate Josephson effects \cite{PhysRevA.57.2030,PhysRevA.57.R28}). Here, normally very fine grids are required to resolve the oscillations, whereas our strategy still yields good approximations in low dimensional spaces and, hence, reduces the costs within the iteration procedure tremendously.

\section{Model problem}\label{section-model}
Consider the dimensionless Gross-Pitaevskii equation in some bounded Lipschitz domain $\Omega\subset\mathbb{R}^d$ where $d=1,2,3$. Since ground state solutions show an extremely fast decay (typically exponential), the restriction to bounded domains and homogeneous Dirichlet condition are physically justified. We seek (in the sense of distributions) the minimal eigenvalue $\lambda$ and corresponding $L^2$-normalized eigenfunction $u\in H^1_0(\Omega)$ with
\begin{equation*}
  \begin{aligned}
    -\ddiv A\nabla u + b u + \beta |u|^2 u&=  \lambda u \quad \text{in }\Omega, \\
    u & = 0 \quad \hspace{8pt}\text{on }\partial\Omega.
  \end{aligned}
\end{equation*}
The underlying data satisfies the following assumptions:
\begin{itemize}
\item[(a)] If $d=1$, the domain $\Omega$ is an interval. If $d=2$ (resp. $d=3$), $\Omega$ has a polygonal (resp. polyhedral) boundary. 
\item[(b)] The diffusion coefficient $A\in L^\infty(\Omega,\mathbb{R}^{d\times d}_{sym})$ is a symmetric matrix-valued function with uniform spectral bounds $\gmax\geq\gmin>0$, 
\begin{equation}\label{e:spectralbound}
\sigma(A(x))\subset [\gmin,\gmax]\quad\text{for almost all }x\in \Omega.
\end{equation}
\item[(c)] $b \in L^2(\Omega)$ is non-negative (almost everywhere). 
\item[(d)] $\beta\in \mathbb{R}$ is non-negative.
\end{itemize}
The weak solution of the GPE minmizes the energy functional $E\hspace{-2pt}: \hspace{-2pt} H^1_0(\Omega) \hspace{0pt} \hspace{-2pt}\rightarrow \hspace{-2pt} \R$ given by
\begin{align*}
E(\phi):= \frac{1}{2} \int_{\Omega} A \nabla \phi \cdot \nabla \phi \dx + \frac{1}{2} \int_{\Omega} b \phi^2 \dx + \frac{1}{4} \int_{\Omega} \beta |\phi|^4 \dx\quad\text{for }\phi \in H^1_0(\Omega).
\end{align*}
\begin{problem}[Weak formulation of the Gross-Pitaevskii equation]
\label{weak-problem}$\\$
Find $u\in H^1_0(\Omega)$ such that $u\ge0$ a.e. in $\Omega$, $\|u\|_{L^2(\Omega)}=1$, and
\begin{align*}
E(u)=\underset{\|v\|_{L^2(\Omega)}=1}{\inf_{v \in H^1_0(\Omega)}} E(v).
\end{align*}
\end{problem}
It is well-known (see, e.g., \cite{Lieb:Seiringer:Yngvason:2000} and \cite{Cances:Chakir:Maday:2010}) that there exists a unique solution $u\in H^1_0(\Omega)$ of Problem~\ref{weak-problem}. This solution $u$ is continuous in $\bar{\Omega}$ and positive in $\Omega$. 
The corresponding eigenvalue $\lambda:=2 E(u) + 2^{-1} \beta \|u\|_{L^4(\Omega)}^4$ of the GPE is real, positive, and simple.
Observe that the eigenpair $(u,\lambda)$ satisfies 
\begin{align*}
\int_{\Omega} A \nabla u \cdot \nabla \phi \dx + \int_{\Omega} b u \phi \dx + \int_{\Omega} \beta |u|^2 u \phi \dx = \lambda \int_{\Omega} u \phi \dx
\end{align*}
for all $\phi \in H^1_0(\Omega)$. Moreover, $\lambda$ is the smallest amongst all possible eigenvalues and satisfies the a priori bound $\lambda < 4E(u)$.

\section{Discretization}
\label{section-discretizations}
This section recalls classical finite element discretizations and presents novel two-grid approaches for the numerical solution of Problem~\ref{weak-problem}. The existence of a minimizer of the functional $E$ in discrete spaces is easily seen. However, uniqueness does not hold in general. We note that unlike claimed in \cite{Zhou:2004} the uniqueness proof given in \cite{Lieb:Seiringer:Yngvason:2000} does not generalize to arbitrary subspaces of the original solution space.

\begin{remark}[Existence of discrete solutions \cite{Cances:Chakir:Maday:2010}]
\label{existence-discrete-solutions}Let $W$ denote a finite dimensional, non-empty subspace of $H^1_{0}(\Omega)$, then there exists a minimizer $u_W\in W$ with $\|u_W\|_{L^2(\Omega)}=1$, $(u_W,1)_{L^2(\Omega)}\ge0$, and
\begin{align*}
E(u_W)=\underset{\|w\|_{L^2(\Omega)}=1}{\inf_{w \in W}} E(w).
\end{align*}
If $(W_i)_{i \in \mathbb{N}}$ represents a dense family of such subspaces, then any sequence of corresponding minimizers $(u_i)_{i \in \mathbb{N}}$ with $(u_i,1)_{L^2(\Omega)}\ge0$ converges to the unique solution $u$ of Problem~\ref{weak-problem}. 
\end{remark}

\subsection{Standard Finite Elements}
\label{subsection-fem-discretization}
We consider two regular simplicial meshes $\T_H$ and $\T_h$ of $\Omega$. The finer mesh $\T_h$ is obtained from  the coarse mesh $\T_H$ by regular mesh refinement. The discretization parameters $h\leq H$ represent the mesh size, i.e., $h_T:=\dia(T)$ (resp. $H_T:=\dia(T)$) for $T \in \T_h$ (resp. $\T_H$) and $h:=\max_{T\in \T_h}\{h_T\}$ (resp. $H:=\max_{T\in \T_H}\{H_T\}$). For $\T=\T_H,\T_h$, let
\begin{equation*}
P_1(\T) = \{v \in L^2(\Omega) \;\vert \;\forall T\in\T,v\vert_T \text{ is a polynomial of total degree}\leq 1\}
\end{equation*}
denote the set of $\T$-piecewise affine functions. Classical $H^1_0(\Omega)$-conforming finite element spaces are then given by
\begin{align*}
V_h:=P_1(\T_h)\cap H^1_0(\Omega) \quad \text{and} \quad V_H:=P_1(\T_H)\cap H^1_0(\Omega) \subset V_h.
\end{align*}
Note that on the fine discretization scale, a different choice of polynomial degree, e.g., piecewise quadratic functions, is possible. This would be a better choice for smooth data that allows for a regular ground state. Our method and its analysis essentially require the inclusion  $H^1_0(\Omega)\supset V_h\supset V_H$. 
The discrete problem on the fine grid $\T_h$ reads as follows.
\begin{problem}[Reference finite element discretization on the fine mesh]
\label{discrete-problem-h}$\\$
Find $u_h\in V_h$ with $(u_h,1)_{L^2(\Omega)}\ge0$, $\|u_h\|_{L^2(\Omega)}=1$ and
\begin{align}
\label{minimizer-in-V-h}E(u_h)=\underset{\|v_h\|_{L^2(\Omega)}=1}{\inf_{v_h \in V_h}} E(v_h).
\end{align}
The corresponding eigenvalue is given by $\lambda_h:=2 E(u_h) + 2^{-1} \beta \| u_h \|_{L^4(\Omega)}^4$. 
\end{problem}

According to Remark~\ref{existence-discrete-solutions}, $u_h$ is not determined uniquely in general. Moreover, $\lambda_h$ is not necessarily the smallest eigenvalue of the corresponding discrete eigenvalue problem. 
In what follows, $u_h$ refers to an arbitrary solution of Problem~\ref{discrete-problem-h}. It will serve as a reference to compare further (cheaper)  numerical approximations with. The accuracy of $u_h$ has been studied in \cite{Cances:Chakir:Maday:2010}. Under the assumption of sufficient regularity, optimal orders of convergence are obtained (cf. \eqref{e:erroruh}).

\subsection{Preprocessing motivated by numerical homogenization}
\label{subsection-two-grid-discretization}
The aim of this paper is to accurately approximate the finescale reference solution $u_h$ of Problem~\ref{discrete-problem-h} within some low-dimensional subspace of $V_h$. For this purpose, we introduce a two-grid upscaling discretization that was initially proposed in \cite{Malqvist:Peterseim:2011} for the treatment of multiscale problems. The framework has been applied to non-linear problems in \cite{HMP12}, to linear eigenvalue problems in \cite{Malqvist:Peterseim:2012} and in the context of Discontinuous Galerkin \cite{EGMP13} and Partition of Unity Methods \cite{HMP13}. This contribution aims to generalize and analyze the methodology to the case of an eigenvalue problem with an additional nonlinearity in the eigenfunction. We emphasize that the co-existence of two difficulties, the nonlinear nature of the eigenproblem itself and the additional nonlinearity in the eigenfunction, requires new essential ideas far beyond simply plugging together existing theories for the isolated difficulties. 

Let $\mathcal{N}_H$ denote the set of interior vertices in $\T_H$. For $z \in \mathcal{N}_H$ we let $\Phi_z \in V_H$ denote the corresponding nodal basis function with $\Phi_z(z)=1$ and $\Phi_z(y)=0$ for all $y\in \mathcal{N}_H \setminus \{ z\}$. We define a weighted Cl\'ement-type interpolation operator (c.f. \cite{MR1736895})
\begin{align}
\label{def-weighted-clement} I_H : H^1_0(\Omega) \rightarrow V_H,\quad v\mapsto I_H(v):= \sum_{z \in \mathcal{N}_H} v_z \Phi_z \quad \text{with }v_z := \frac{(v,\Phi_z)_{L^2(\Omega)}}{(1,\Phi_z)_{L^2(\Omega)}}.
\end{align}

It is easily shown by Friedrichs' inequality and the Sobolev embedding $H^1_0(\Omega) \hookrightarrow L^6(\Omega)$ (for $d\le3$) that 
\begin{align*}
a(v,\phi):=\int_{\Omega} A \nabla v \cdot \nabla \phi \dx + \int_{\Omega} b v \phi \dx\qquad \text{for }v,\phi\in H^1_0(\Omega)
\end{align*}
defines a scalar product in $H^1_0(\Omega)$ and induces a norm $\tnorm{\cdot}:=\sqrt{a(\cdot,\cdot)}$ on $H^1_0(\Omega)$ which is equivalent to the standard $H^1$-norm. 
By means of the interpolation operator $I_H$ defined in \eqref{def-weighted-clement}, we construct an $a$-orthogonal decomposition of the space $V_h$ into a low-dimensional coarse space $\Vc$ (with favorable approximation properties) and a high-dimensional residual space $\Vf$. The residual or 'fine' space is the kernel of the interpolation operator restricted to $V_h$, 
\begin{subequations}\label{e:decomposition}
\end{subequations}
\begin{align}\label{e:decompositiona}
\Vf := \operatorname{kernel}(I_H\vert_{V_h}).\tag{\ref{e:decomposition}.a}
\end{align}
The coarse space is simply defined as the orthogonal complement of $\Vf$ in $V_h$ with respect to $a(\cdot,\cdot)$. 
It is characterized via the $a$-orthogonal projection $\Pf : H^1_0(\Omega) \rightarrow \Vf$ onto the fine space given by
\begin{align*}
a(\Pf v,\phi) = a(v,\phi)\quad\text{for all }\phi \in \Vf.
\end{align*}
By defining $\Pc:=1-\Pf$, the coarse space is given by
\begin{align}\label{e:decompositionb}
\Vc := \Pc V_H.\tag{\ref{e:decomposition}.b}
\end{align}
A basis of $\Vc$ is given by $\left( \Pc \Phi_z\right)_{z\in\mathcal{N}_H}$ with $\dimension\Vc=\operatorname{dim} V_H$. With this definition we obtain the splitting
\begin{align}\label{e:decompositionc}
V_h = \Vc \oplus \Vf.\tag{\ref{e:decomposition}.c}
\end{align}
Some favorable properties of the decomposition, in particular its $L^2$-quasi-orthogonality, are discussed in Section~\ref{subsection-properties-of-the-decomposition}. 
The minimization problem in the low-dimensional space $\Vc$ reads as follows.
\begin{problem}[Pre-processed approximation]
\label{discrete-problem-in-V-c}$\\$
Find $u^\cs_H\in \Vc$ with $(u^\cs_H,1)\ge0$, $\|u^\cs_H\|_{L^2(\Omega)}=1$ and
\begin{align*}
E(u^\cs_H)=\underset{\|v^\cs\|_{L^2(\Omega)}=1}{\inf_{v^\cs \in \Vc}} E(v^\cs).
\end{align*}
The corresponding eigenvalue in $\Vc$ is given by
$\lambda^\cs_H:=2 E(u^\cs_H) + 2^{-1} \beta \| u^\cs_H \|_{L^4(\Omega)}^4$.
\end{problem}
\begin{remark}[Practical aspects of the decomposition] \begin{itemize}\item[a)]The assembly of the corresponding finite element matrices requires only the evaluation of $\Pf\Phi_z$, i.e., the solution to one linear Poisson-type problem per coarse vertex. This can be done in parallel. Section~\ref{ss:sparse} below will show that these linear problems may be restricted to local subdomains centered around the coarse vertices without loss of accuracy. Hence, even in a serial computing setup, the complexity of solving all corrector problems is equivalent (up to factor $|\log(H)|$) to the cost of solving one linear Poisson problem on the fine mesh. 
\item[b)] The pre-processing step is independent of the parameter $\beta$ which characterizes the species of the bosons. Hence, the method becomes considerably cheaper when experiments need to be carried out for different types and numbers of bosons. A similar argument applies to variations on the trapping potential $b$. Provided that this trapping potential is an element of $H^1(\Omega)$ (in practical applications it is usually even harmonic and admits the desired regularity) the bilinear form $a(\cdot,\cdot)$ (and the associated constructions of $\Vf$ and $\Vc$) can be restricted to the second order term {$\int_{\Omega} A \nabla v \cdot \nabla \phi$} without a loss in the expected convergence rates stated in Theorems~\ref{main-result-pre} and~\ref{main-result-post} below. The trapping potential may then be varied without affecting the pre-processed space $\Vc$.
\item[c)] Once the coarse space has been assembled it can also be re-used in computations of larger eigenvalues (i.e., not only in the ground state solution).
\end{itemize}
\end{remark}

\subsection{Sparse approximations of $\Vc$}\label{ss:sparse}
The construction of the coarse space $\Vc$ is based on fine scale equations formulated on the whole domain $\Omega$ which makes them  expensive to compute. However, \cite{Malqvist:Peterseim:2011} shows that $\Pf\Phi_z$ decays exponentially fast away from $z$. We specify this feature as follows. Let $k\in\mathbb{N}$ denote the localization parameter, i.e., a new discretization parameter. We define nodal patches $\omega_{z,k}$ of $k$ coarse grid layers centered around the node $z\in\mathcal{N}_H$ by
\begin{equation}\label{e:omega}
 \begin{aligned}
 \omega_{z,1}&:=\support \Phi_z=\cup\left\{T\in\T_H\;|\;z\in T\right\},\\
 \omega_{z,k}&:=\cup\left\{T\in\T_H\;|\;T\cap \omega_{z,k-1}\neq\emptyset\right\} \quad \mbox{for} \enspace k\geq 2.
\end{aligned}
\end{equation}
There exists $0 < \theta < 1$ depending on the contrast $\gmin/\gmax$ but not on mesh sizes $h,H$ and fast oscillations of $A$ such that for all 
for all vertices $z\in\mathcal{N}_H$ and for all $k\in\mathbb{N}$, it holds
\begin{equation}\label{l:decay}
\| \Pf\Phi_z \|_{H^1(\Omega\setminus\omega_{z,k})}\lesssim \theta^{k}\tnorm{\Pf\Phi_z}.
\end{equation}  
This result motivates the truncation of the computations of the basis functions to local patches $\omega_{z,k}$. We approximate $\Psi_z=\Pf\Phi_z\in\Vf$ from \eqref{e:decompositiona}-\eqref{e:decompositionc} with $\Psi_{z,k}\in\Vf(\omega_{z,k}):=\{v\in\Vf\;\vert\;v\vert_{\Omega\setminus\omega_{x,k}}=0\}$ such that
\begin{equation}\label{e:correctorlocal}
 a(\Psi_{z,k},v)=a(\Phi_z,v)\quad\text{for all }v\in\Vf(\omega_{z,k}).
\end{equation}
This yields a modified coarse space $V^{\cs}_{H,h,k}$ with a local basis
\begin{equation}\label{e:basisVck}
 V^{\cs}_{H,h,k} = \operatorname{span}\{\Phi_z-\Psi_{z,k}\;\vert\;z\in\mathcal{N}_H\}.
\end{equation}
The number of non-zero entries of the corresponding finite element matrices is proportional to $k^d N_H$ (note that we expect $N_H^2$ non-zero entries without the truncation). Due to the exponential decay, the very weak condition $k\approx |\log{H}|$ implies that the perturbation of the ideal method due to this truncation is of higher order and forthcoming error estimates in Theorems~\ref{main-result-pre} and~\ref{main-result-post} remain valid. We refer to \cite{Malqvist:Peterseim:2011} for details and proofs. The modified localization procedure from \cite{HP12} with improved accuracy and stability properties may also be applied.

\subsection{Post-processing}
Although $u^\cs_H$ and $\lambda^\cs_H$ will turn out to be highly accurate approximations of the unknown solution $(u,\lambda)$, the orders of convergence can be improved even further by a simple post-processing step on the fine grid. The post-processing applies the two-grid method originally introduced by Xu and Zhou \cite{Xu:Zhou:2001} for linear elliptic eigenvalue problems to the present equation by using our upscaled coarse space on the coarse level.
\begin{problem}[Post-processed approximation]
\label{two-grid-postprocessing}
Find $u^\cs_h\in V_h$ with
\begin{align*}
\int_{\Omega} A \nabla u^\cs_h \cdot \nabla \phi_h \dx + \int_{\Omega} b u^\cs_h \phi_h \dx = \lambda^\cs_H \int_{\Omega} u^\cs_H \phi_h \dx - \int_{\Omega} \beta |u^\cs_H|^2 u^\cs_H \phi_h \dx
\end{align*}
for all $\phi_h \in V_h$.  Define $\lambda^\cs_h:=(2 E(u^\cs_h) + 2^{-1} \beta \| u^\cs_h \|_{L^4(\Omega)}^4) \|u^\cs_h\|^{-2}_{L^2(\Omega)}$.
\end{problem}
Let us emphasize that this approach is different from \cite{Chien:Huang:Jeng:Li:2008}, where the post-precessing problem has a different structure and where classical finite element spaces are used on both scales. 
\section{A-priori error estimates}
\label{section-main-results}
This section presents the a-priori error estimates for the pre-processed/upscaled approximation with and without the post-processing step. Throughout this section, $u \in H^1_0(\Omega)$ denotes the solution of Problem~\ref{weak-problem}, $u_h \in V_h$ the solution of reference Problem~\ref{discrete-problem-h}, $u_H^\cs \in \Vc$ the solution of Problem~\ref{discrete-problem-in-V-c} and $u_h^\cs$ the post-processed solution of Problem~\ref{two-grid-postprocessing}. 
The notation $f \lesssim g$ abbreviates $f\leq Cg$ with some constant $C$ that may depend on the space dimension $d$, $\Omega$, $\gmin$, $\gmax$, $\|b\|_{L^2(\Omega)}$, $\beta$, $\lambda$ and interior angles of the triangulations, but not on the mesh sizes $H$ and $h$. In particular it is robust against fast oscillations of $A$ and $b$. 

\begin{theorem}[Error estimates for the pre-processed approximation]
\label{main-result-pre}
Assume that $\| u - u_h \|_{H^1(\Omega)}\lesssim 1$. For $u$ and $u_H^\cs$ as above, it holds
\begin{align}
\label{final-H1-estimate}\| u - u_H^\cs \|_{H^1(\Omega)} &\lesssim H^2 + \| u - u_h \|_{H^1(\Omega)}.
\end{align}
For sufficiently small $h$ (in the sense of Canc\`es et al. \cite{Cances:Chakir:Maday:2010}), we also have
\begin{align}
\label{final-L2-estimate}|\lambda- \lambda_H^\cs|+\| u - u_H^\cs \|_{L^2(\Omega)} &\lesssim H^3 + H\hspace{2pt}\| u - u_h \|_{H^1(\Omega)}.
\end{align}
\end{theorem}
\begin{proof} The proof is postponed to Section~\ref{proof-main-result-pre}. \end{proof}

The additional post-processing improves, roughly speaking, the order of accuracy by one. 
\begin{theorem}[Error estimates for the post-processed approximation]
\label{main-result-post}
$\\$
Assume that $h$ is sufficiently small. The post-processed approximation $u_h^\cs$ and the post-processed eigenvalue $\lambda_h^\cs$ satisfy:
\begin{align}
\label{final-post-processed-H1-estimate}\|u -u^\cs_h\|_{H^1(\Omega)} &\lesssim H^3 + \| u - u_h \|_{H^1(\Omega)},\\
\label{final-post-processed-L2-estimate}| \lambda - \lambda^\cs_h |+\| u - u^\cs_h \|_{L^2(\Omega)} &\lesssim H^4 + C_{L^2}(h,H).
\end{align}
The constant $C_{L^2}(h,H)$ behaves roughly like $H^2 \| u - u_h \|_{H^1(\Omega)}$ and can be extracted from the proofs in Section~\ref{subsubsection-postprocessed-L2}. 
\end{theorem}
\begin{proof} The proof is postponed to Section~\ref{proof-main-result-post}. \end{proof}

Let us emphasize that both theorems remain valid for $\Vc$ replaced with its sparse approximation $V^{\cs}_{H,h,k}$ (cf. Section~\ref{ss:sparse}) for moderate localization parameter $k\gtrsim |\log H|$. 

We shall discuss the behavior of the finescale errors $u-u_h$ and $\lambda-\lambda_h$. Recall from \cite{Cances:Chakir:Maday:2010} that for a bounded domain $\Omega$ with polygonal Lipschitz-boundary, $A\in [W^{1,\infty}(\Omega)]^{d \times d}$, and  {\rm sufficiently small} $h$, the fine scale error $\| u - u_h \|_{H^1(\Omega)}$ satisfies the optimal estimate
\begin{align}\label{e:erroruh}
\| u - u_h \|_{H^1(\Omega)} + h^{-1} \| u - u_h \|_{L^2(\Omega)} + h^{-1}|\lambda - \lambda_h|  &\lesssim h.
\end{align}
The proof in \cite{Cances:Chakir:Maday:2010} is for constant $A=1$ and hyperrectangle $\Omega$ but it is easily checked that the estimates remain valid for any bounded domain $\Omega$ with polygonal Lipschitz-boundary and $A\in [W^{1,\infty}(\Omega)]^{d \times d}$. 
Under these assumptions our a priori estimates for the post-processed approximation of the ground state eigenvalue summarize as follows
\begin{equation*}
 | \lambda - \lambda^\cs_h | \lesssim H^4 + H^2  h.
\end{equation*}
Hence, in this regular setting, the choice $H=h^{1/2}$ ensures that the loss of accuracy is negligible when compared to the accuracy of the expensive full fine scale approximation $\lambda_h$. However, with regard to the numerical experiment in Section~\ref{ss:numexp1} below, this choice might be pessimistic. 

Moreover, note that the fine scale error depends crucially on higher Sobolev regularity of the solution whereas our estimates for the coarse scale error require only minimal regularity that holds under the assumption (a)--(d) in Section~\ref{section-model}. Thus, we believe that in a less regular setting, even coarser choices of $H$ relative to $h$ will balance the discretization errors on the coarse and the fine scale.

\section{Numerical experiments}
\label{section-numerical-experiments}

Any numerical approach for the computation of ground states of a BEC involves an iterative algorithm that starts with a given initial value and diminishes the energy of the density functional $E$ in each iteration step. In this contribution, we use the Optimal Damping Algorithm (ODA) originally developed by Canc\`es and Le Bris \cite{Cances:LeBris:2000,Cances:2000} for the Hartree-Fock equations, since it suits our pre-processing framework. The ODA involves solving a linear eigenvalue problem in each iteration step. However, after pre-processing these linear eigenvalue problems are very low dimensional and the precomputed basis of $\Vc$ can be reused for each of these problems making the iterations extremely cheap. The approximations produced by the ODA are known to rapidly converge to a solution of the discrete minimization problem (see \cite{Dion:Cances:2007} and \cite{Cances:2000} for a proof in the setting of the Hartree-Fock equations).
All subsequent numerical experiments have been performed using MATLAB.

\subsection{Numerical results for harmonic potential}\label{ss:numexp1}
In this section, we choose the smooth experimental setup of \cite[Section~4, p. 109 and Fig.~2 (bottom)]{Cances:Chakir:Maday:2010}, i.e., 
$\Omega:=(0,\pi)^2$, $b(x_1,x_2):=x_1^2+x_2^2$, $A=1$, $\beta=1$ and with homogeneous Dirichlet boundary condition. Our method depends basically on three parameters, the coarse mesh size $H$, the fine mesh size $h$, and the localization parameter $k$ (cf. Section~\ref{ss:sparse} and \cite{HP12}). In all computations of this section we couple $k$ to the coarse mesh size by choosing $k=2\log_2 H$. This choice is made such that the error of localization is negligible when compared with the errors committed be the fine scale discretization and the upscaling. All approximations are computed with the ODA method as presented in \cite[Section~2]{Dion:Cances:2007} with accuracy parameter $\varepsilon_{\operatorname{ODA}}=10^{-14}$. 

\subsubsection{Comparison with full fine scale approximation}\label{ss:numexp1:1}
\begin{figure}[t]
\includegraphics[width=0.495\textwidth]{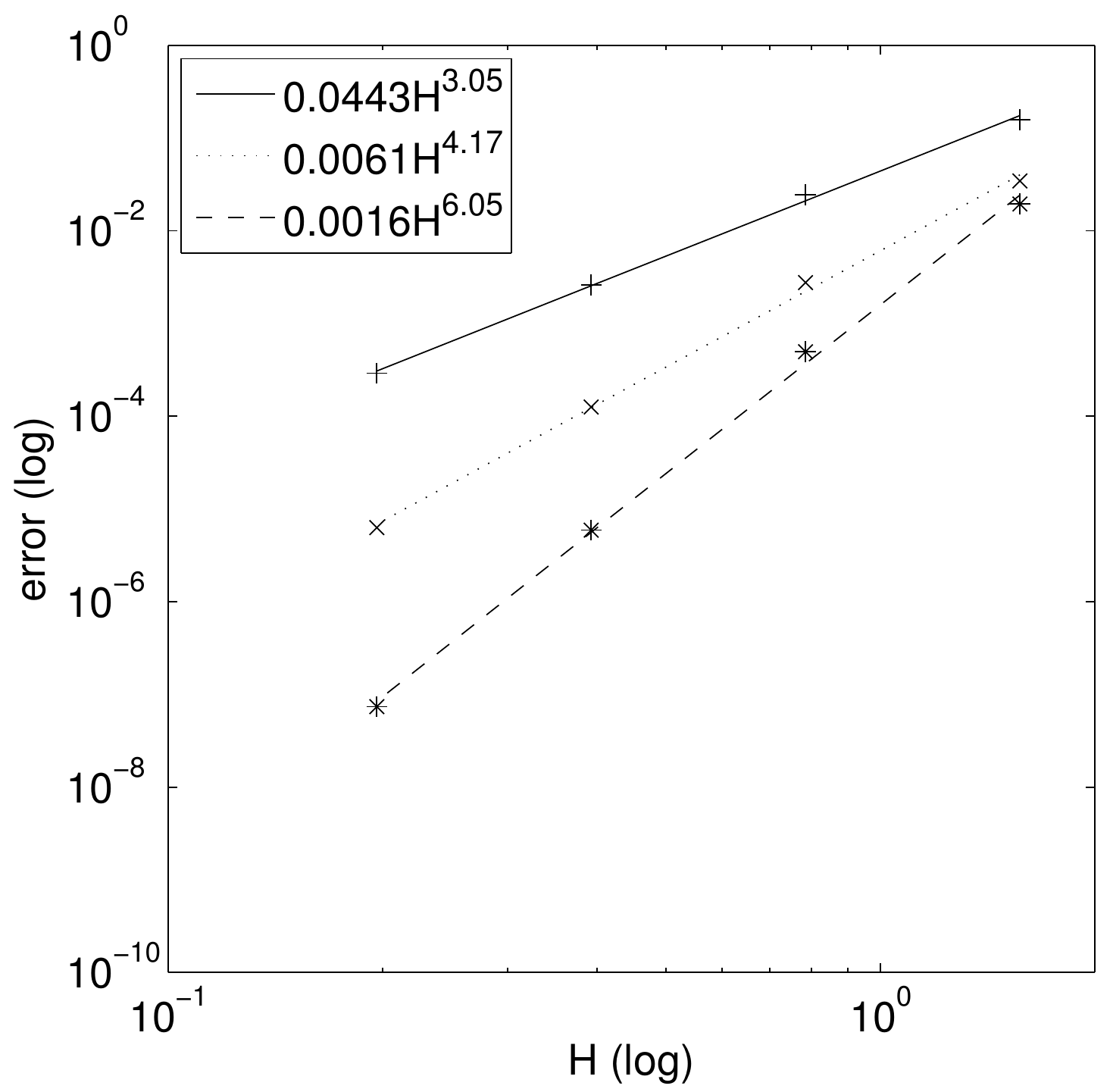}
\includegraphics[width=0.495\textwidth]{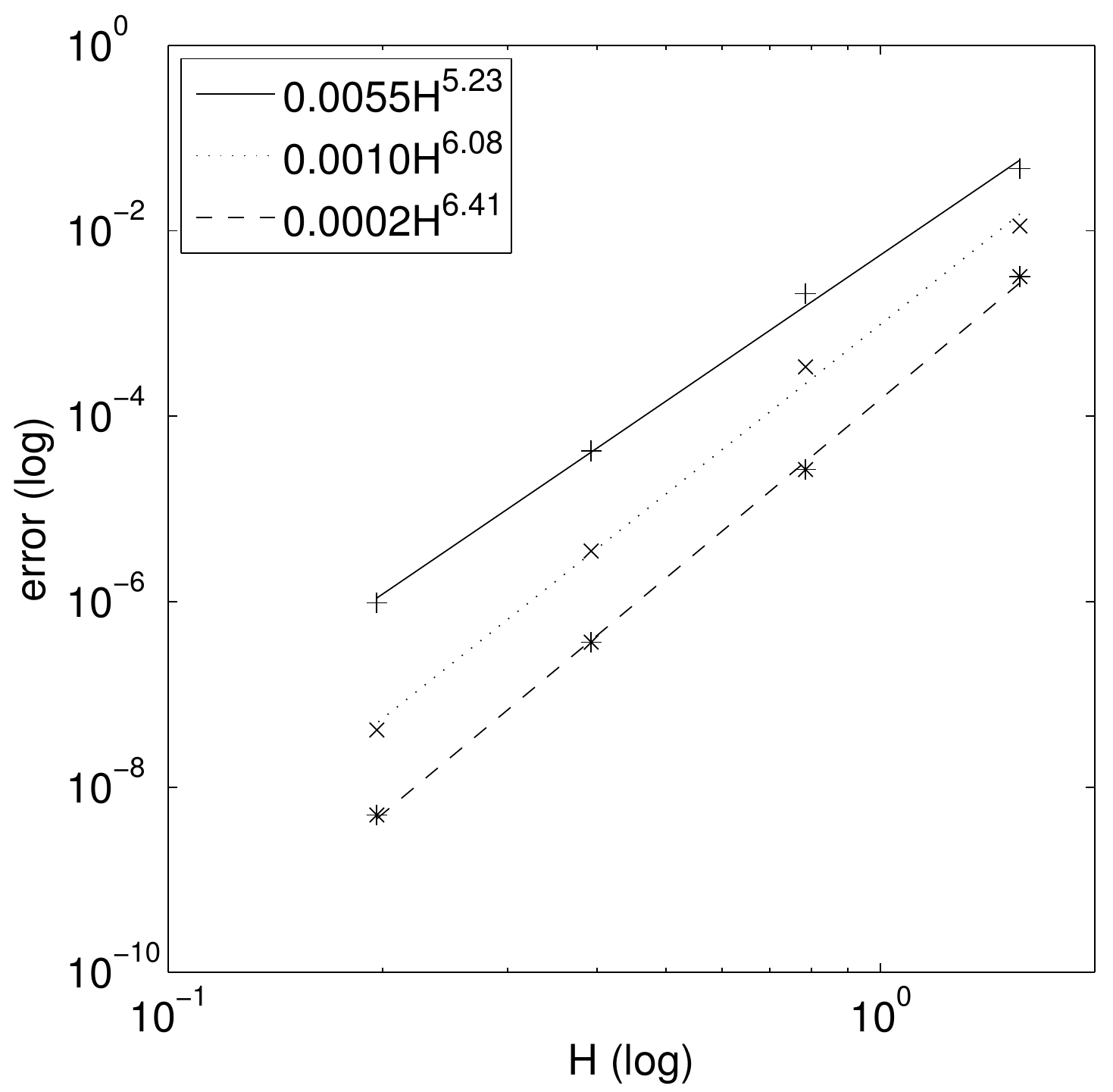}
\caption{Results for harmonic potential. Left: Errors of pre-processed approximation $\| u_h - u_H^\cs \|_{H^1(\Omega)}$ ($+$), {$\| u_h - u_H^\cs \|_{L^2(\Omega)}$} ($\times$), and $| \lambda_h - \lambda_H^\cs |$ ($*$) vs. coarse mesh size $H$. 
Right: Errors of post-processed approximation $\| u_h - u_h^\cs \|_{H^1(\Omega)}$ ($+$), $\| u_h - u_h^\cs \|_{L^2(\Omega)}$ ($\times$), and $| \lambda_h - \lambda_h^\cs |$ ($*$) vs. coarse mesh size $H$.}\label{fig:harmonic}
\end{figure}
In the first experiment, we consider uniform coarse meshes $\T_H$ with mesh width parameters $H=2^{-1}\pi,2^{-2}\pi,\ldots,2^{-4}\pi$ of $\Omega$. The fine mesh $\T_h$ for the pre- and post-processing has width $h=2^{-7}\pi$ and remains fixed. We study the error committed by coarsening from a fine scale $h$ to several coarse scales $H$, i.e., we study the distance between the ground state $(u_h,\lambda_h)$ of Problem~\ref{discrete-problem-h} and either the coarse scale approximation $(u^\cs_H,\lambda^\cs_H)$ of Problem~\ref{discrete-problem-in-V-c} (with underlying finescale $h$) or its post-processed version $(u^\cs_h,\lambda^\cs_h)$ of Problem~\ref{two-grid-postprocessing}. 
Our theoretical results do not allow predictions about the coarsening error. Most likely, this is an artifact of our theory and we conjecture that $(u_h,\lambda_h)$ and its coarse approximations $(u^\cs_H,\lambda^\cs_H)$ and $(u^\cs_h,\lambda^\cs_h)$ are in fact super-close in the sense of 
\begin{equation}\label{e:superclose}
\begin{aligned}
 H^{-1}\| u_h - u_H^\cs \|_{H^1(\Omega)}+\| u_h - u_H^\cs \|_{L^2(\Omega)}+| \lambda_h - \lambda^\cs_H | &\lesssim H^3,\\
 H^{-1}\| u_h - u_h^\cs \|_{H^1(\Omega)}+\| u_h - u_h^\cs \|_{L^2(\Omega)}+| \lambda_h - \lambda^\cs_h | &\lesssim H^4.
\end{aligned}
\end{equation}
This assertion is true in the limit $h\rightarrow 0$. Section~\ref{ss:numexphigh} supports numerically the assertion for positive $h$. 
Figure~\ref{fig:harmonic} reports the numerical results. Observe that the experimental rates with respect to $H$ displayed in the figures are in fact better than the rates indicated by Theorems~\ref{main-result-pre}--\ref{main-result-post} and conjectured in  \eqref{e:superclose}.
The reason could be the high regularity of the underlying (exact) solution $u\in H^3(\Omega)$. We do not exploit additional regularity in our error analysis. Similar observations have been made for the linear eigenvalue problem; see
\cite[Remark 3.3]{Malqvist:Peterseim:2012} for details and some justification of higher rates under additional regularity assumptions.
Our implementation is not yet adequate for a fair comparison with regard to computational complexity and computing times between standard fine scale finite elements and our two-level techniques. However, to convince the reader of the potential savings in our new approach, let us mention that the number of iterations of the ODA were basically the same for both approaches in all numerical experiments. This statement applies as well to more challenging setups with larger values of $\beta$ (see, e.g., Section~\ref{ss:numexp2} below) where ODA needs many iterations to fall below some prescribed tolerance. We, hence, conclude that the actual speed-up of our approach is truly reflected by the dimension reduction from $h^{-d}$ to $H^{-d}$ up to the overhead $\mc{O}(k)=\mc{O}(\log{|H|})$ induced by slightly denser (but still sparse) finite element matrices on the coarse level.

\subsubsection{Comparison with high-resolution numerical approximation}\label{ss:numexphigh}
\begin{figure}[t]
\includegraphics[width=0.495\textwidth]{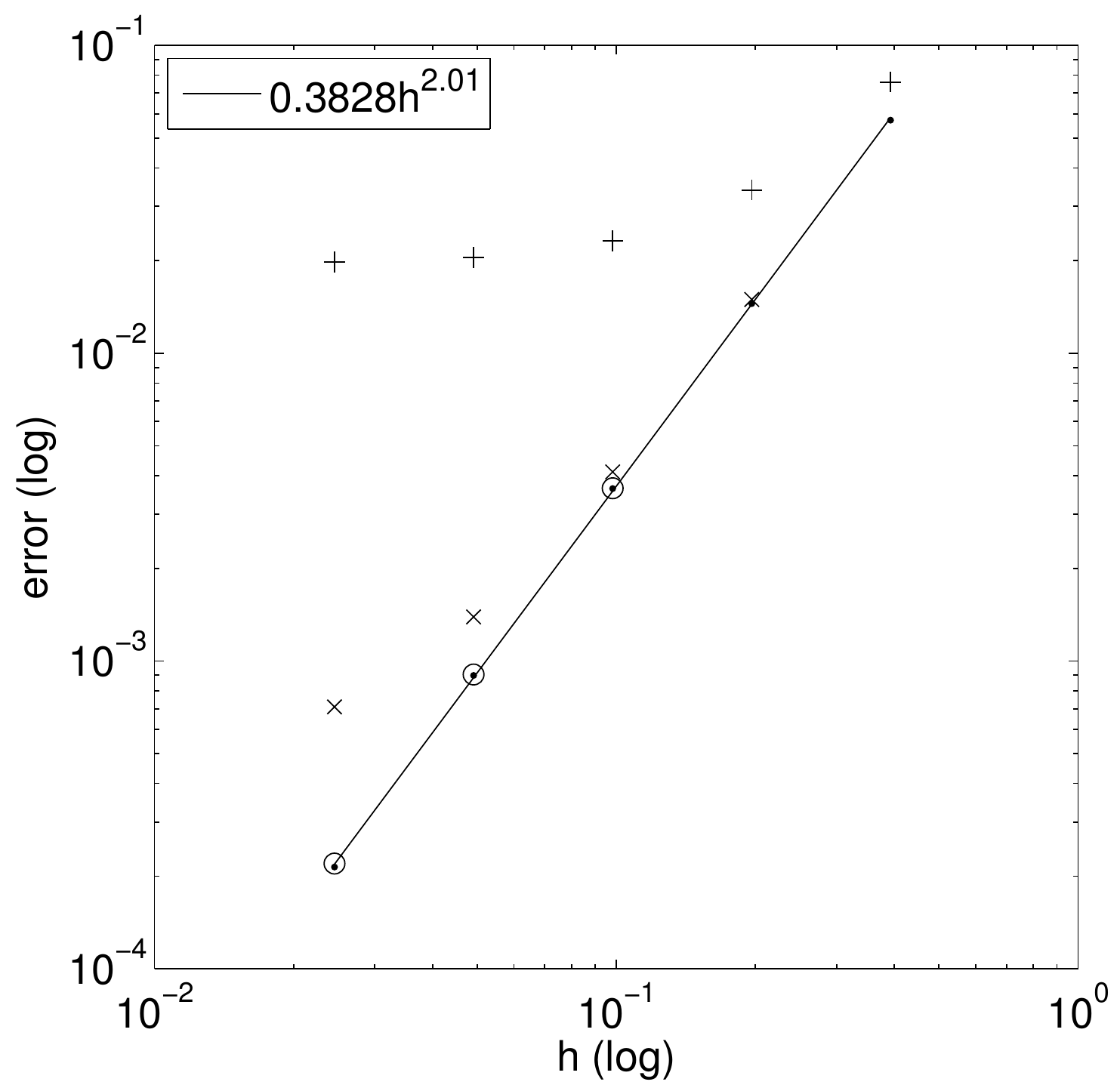}
\includegraphics[width=0.495\textwidth]{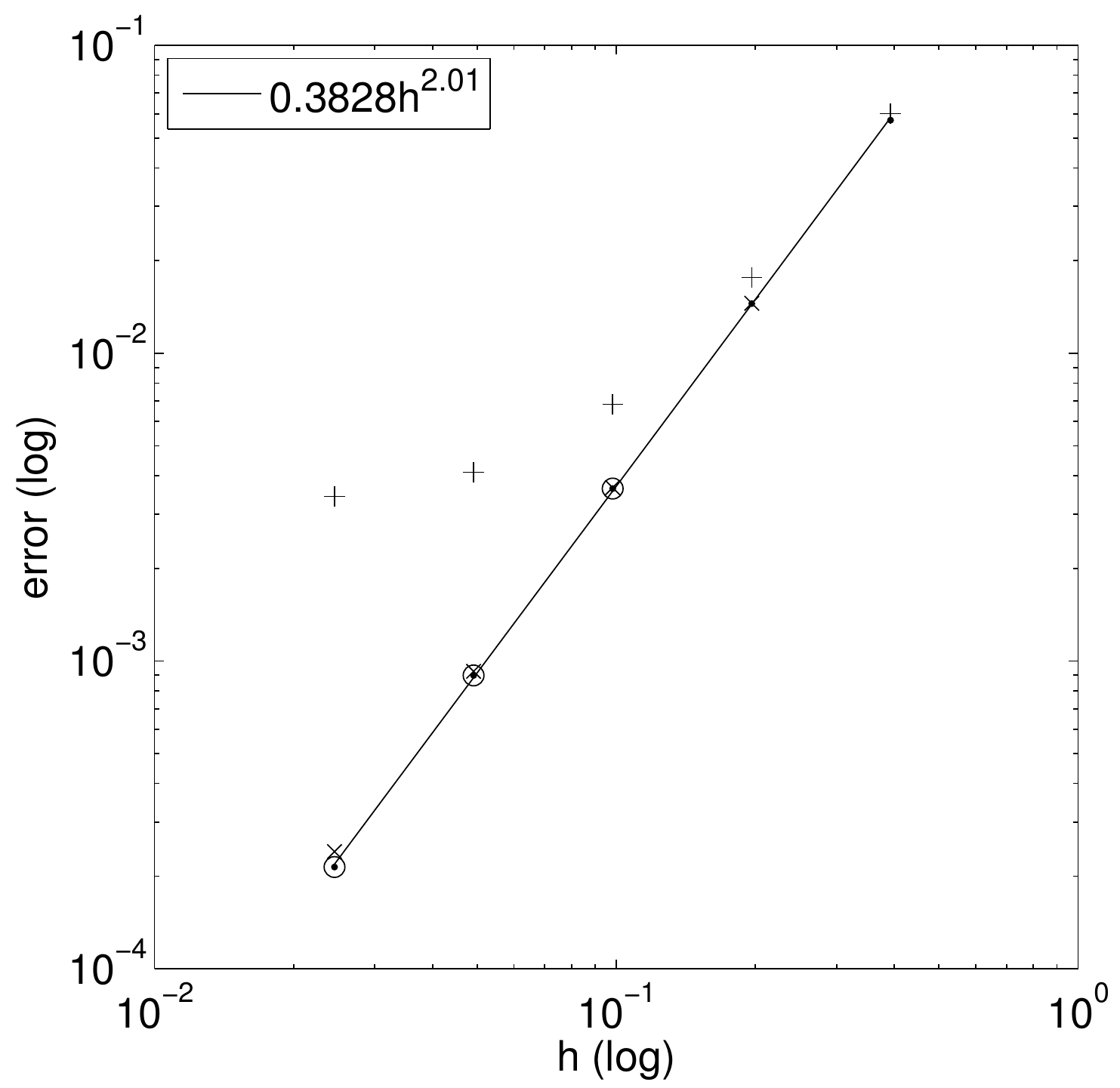}
\caption{Results for harmonic potential. Left: (Estimated) errors of pre-processed approximation $| \lambda - \lambda_H^\cs |$ for fixed values $H=2^{-1}\pi$  ($+$), $H=2^{-2}\pi$ ($\times$) and $H=2^{-3}\pi$ ($\circ$) vs. fine mesh size $h$. 
Right: (Estimated) errors of post-processed approximation $| \lambda - \lambda_{h}^\cs |$ for fixed values $H=2^{-1}\pi$  ($+$), $H=2^{-2}\pi$ ($\times$) and $H=2^{-3}\pi$ ($\circ$) vs. fine mesh size $h$. 
In both plots, the (estimated) error of the standard FEM on the fine mesh $| \lambda - \lambda_h |$ ($\bullet$) is depicted for reference.}\label{fig:harmonich}
\end{figure}
In the second experiment we investigate the role of the fine scale parameter $h$. We consider uniform coarse meshes $\T_H$ with mesh width parameters $H=2^{-1}\pi,\ldots,2^{-3}\pi$ and uniform fine meshes $\T_h$ for $h=H/4,\ldots, 2^{-7}\pi$ for pre- and post-processing computations. The error between the exact eigenvalue $\lambda$ and coarse approximations $\lambda_{H}^\cs$ and $\lambda_{h}^\cs$ is estimated via a high-resolution numerical solution on a mesh of width $2^{-9}\pi$. The results are reported in Figure~\ref{fig:harmonich}. For the sake of clarity, we show eigenvalue errors only. We conclude that it would have been sufficient to choose $H\approx h^{1/3}$ to achieve the accuracy of $\lambda_h$ by our coarse approximation scheme with post-processing. 

\subsection{Numerical results for discontinuous periodic potential}\label{ss:numexp2}
This section addresses the case of a BEC that is trapped in a periodic potential. Periodic potentials are of special interest since they can be used to explore physical phenomena such as Josephson oscillations and macroscopic quantum self-trapping of the condensate (c.f. \cite{PhysRevA.57.2030,PhysRevA.57.R28}). Here we use a potential $b$ that describes a periodic array of quantum wells that can be experimentally generated by the interference of overlapping laser beams (c.f. \cite{Xue:Zhao-Xin:Wei-Dong:2009}).

Let $\Omega=(0,\pi)^2$, $A=1$, and $\beta=4$. Given $b_t=100$ and $L=4$, define
\begin{align*}
b_0(x_1,x_2):=\begin{cases}
0 \quad &\mbox{for} \enspace x \in ]\frac{1}{4},\frac{3}{4}[^2\\
b_t \quad &\mbox{else}
\end{cases}
\end{align*}
and the potential $b(x)=b_0\left(L\left(x/\pi-\frac{\lfloor Lx/\pi\rfloor}{L}\right)\right)$.

Consider the same numerical setup as in Section~\ref{ss:numexp1:1} (i.e., we draw our attention again to the coarsening error $u_h - u_H^\cs$) with the exception that we were able to reduce the localization parameter $k=\log_2 H$ without affecting the best convergence rates possible. Figure~\ref{fig:periodic} reports the errors between the finescale reference discretization and our coarse approximations. 
\begin{figure}[t]
\includegraphics[width=0.495\textwidth]{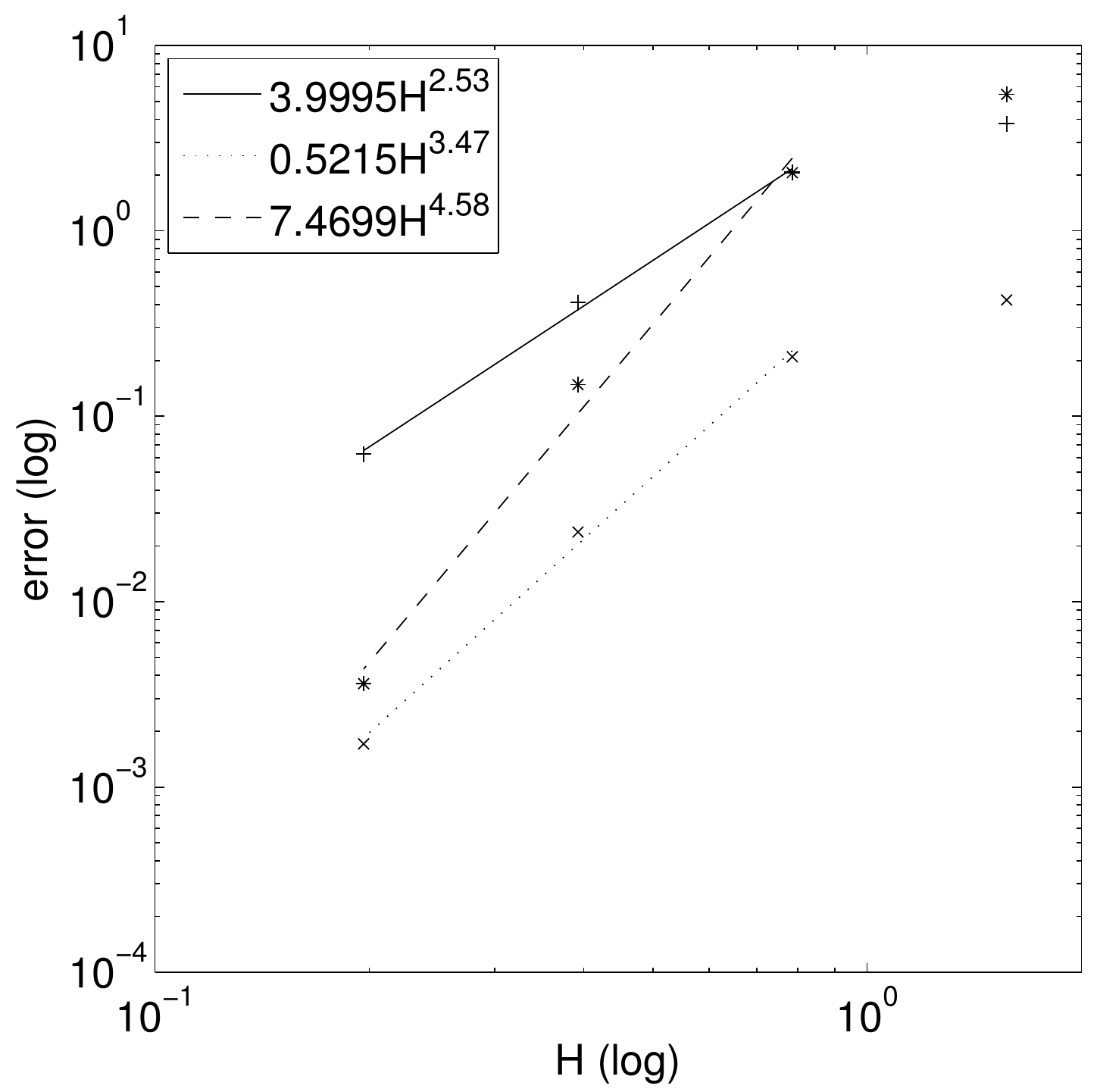}
\includegraphics[width=0.495\textwidth]{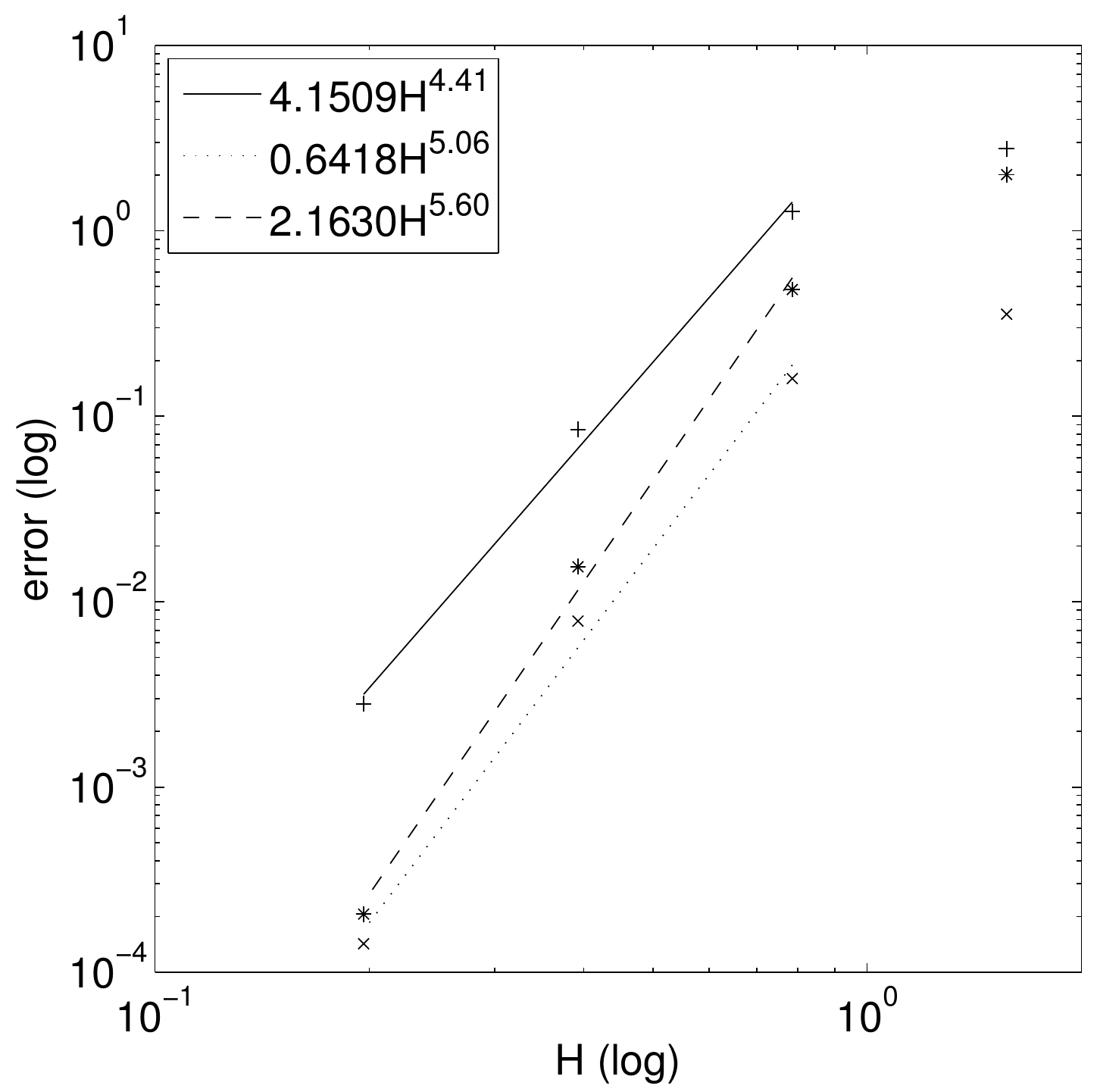}
\caption{Results for periodic potential. Left: Errors of pre-processed approximation $\| u_h - u_H^\cs \|_{H^1(\Omega)}$ ($+$), {$\| u_h - u_H^\cs \|_{L^2(\Omega)}$} ($\times$), and $| \lambda_h - \lambda_H^\cs |$ ($*$) vs. coarse mesh size $H$. 
Right: Errors of post-processed approximation $\| u_h - u_h^\cs \|_{H^1(\Omega)}$ ($+$), $\| u_h - u_h^\cs \|_{L^2(\Omega)}$ ($\times$), and $| \lambda_h - \lambda_h^\cs |$ ($*$) vs. coarse mesh size $H$.}\label{fig:periodic}
\end{figure}
For the discontinuous potential, the experimental rates (with respect to $H$) are slightly worse than those ones observed in Section~\ref{ss:numexp1:1}. However, they are still better than the rates indicated by Theorems~\ref{main-result-pre}--\ref{main-result-post} and conjectured in \eqref{e:superclose}.

\section{Proofs of the main results}
\label{section-proofs-main-results}
In this section we are concerned with proving the main theorems.
\subsection{Auxiliary results}
\label{subsection-auxiliary-estimates}

An application of \cite[Theorem 1]{Cances:Chakir:Maday:2010} shows that $u_h$ and $u^\cs_H$ both converge to $u$ in $H^1(\Omega)$, which guarantees stability.

\begin{remark}[Stability of discrete approximations]
For sufficiently small $h$ we have
\begin{align}
\label{u-energy} \| u_h \|_{H^1(\Omega)} &\le \sqrt{\lambda_h} \lesssim \sqrt{\lambda} \qquad \mbox{and}\\
\label{u-L-4-estimate}|| u_h||_{L^4(\Omega)} &\le \left(\frac{\lambda_h}{\beta}\right)^{\frac{1}{4}} \lesssim \left(\frac{\lambda}{\beta}\right)^{\frac{1}{4}}.
\end{align}
The same results hold for $u_h$ replaced by $u^\cs_H$ and $\lambda_h$ replaced by $\lambda_H^{\cs}$ for $h$ and $H$ sufficiently small.
\end{remark}

The bound \eqref{u-energy} is obvious using $\|u_h\|_{L^2(\Omega)}=1$ and the $H^1$-convergence $u_h \rightarrow u$ which guarantees $\lambda_h \rightarrow \lambda$. Estimate \eqref{u-L-4-estimate} directly follows from the definitions of $\lambda_h$ and $E_h$ which gives us $\lambda_h \ge 2 E(u_h) = a(u_h,u_h) + \frac{\beta}{2} \|u_h\|^4_{L^4(\Omega)} \ge \frac{\beta}{2} \|u_h\|^4_{L^4(\Omega)}$.
\begin{remark}[$L^\infty$-bound]
The solution $u$ of Problem~\ref{weak-problem} is in $L^\infty(\Omega)$. This follows from the uniqueness of $u \in H^1_0(\Omega)$ which shows that it is also the unique solution of the linear elliptic problem
\begin{align*}
\int_{\Omega} A \nabla u \cdot \nabla \phi + b u \phi \dx = \int_{\Omega} \tilde{f} \phi \dx \qquad\text{for all }\phi \in H^1_0(\Omega),
\end{align*}
where $\tilde{f}:= (\lambda u -  \beta |u|^3) \in L^2(\Omega)$. Standard theory for linear elliptic problems (c.f. \cite[Theorem~8.15, pp. 189--193]{Gilbarg}) then yields the existence of a constant $c$ only depending on $\Omega$, $d$ and $\|\gmin^{-1} b\|_{L^{2}(\Omega)}$ such that
\begin{multline}\label{l-infty-estimate}
\| u \|_{L^{\infty}(\Omega)} \le c(\|u\|_{L^2(\Omega)} + \gmin^{-1} \|\tilde{f}\|_{L^2(\Omega)} ) \lesssim 1+ \| u\|_{L^6(\Omega)}^3 \lesssim  1+ \| u \|^3_{H^1(\Omega)}.
\end{multline}
\end{remark}
\subsection{Properties of the coarse space $\Vc$}
\label{subsection-properties-of-the-decomposition}
Recall the local approximation properties of the weighted Cl\'ement-type interpolation operator $I_H$ defined in \eqref{def-weighted-clement},
\begin{align}\label{e:errorclement}
H_T^{-1} \| v-I_H(v) \|_{L^2(T)} + \| \nabla(v-I_H(v)) \|_{L^2(T)} \le C_{I_H} \| \nabla v \|_{L^2(\omega_T)}
\end{align}
for all $v \in H^1_0(\Omega)$. Here, $C_{I_H}$ is a generic constant that depends only on interior angles of $\T_H$ but not on the local mesh size and $\omega_T := \bigcup \{ S \in \mathcal{T}_H| \hspace{2pt} \overline{S} \cap \overline{T} \neq \emptyset \}$.
Furthermore, for all $v \in H^1_0(\Omega)$ and for all $z \in \mathcal{N}_H$ it holds
\begin{align}
\label{local-estimate-clement} \int_{\omega_z} (v-v_z)^2 \dx \le C_{I_H} H^2 \| \nabla v \|_{L^2(\omega_z)}^2,
\end{align}
where $\omega_z := \support(\Phi_z)$ and $v_z$ is given by \eqref{def-weighted-clement}.
\begin{lemma}[Properties of the decomposition]
$\\$
The decomposition of $V_h$ into $V_H$ and $\Vf$ (stated in Section~\ref{subsection-two-grid-discretization})
is $L^2$-orthogonal, i.e.,
\begin{align}
\label{L2-orthogonality}V_h = V_H \oplus \Vf \quad \mbox{and} \quad (v_H,v^\f)_{L^2(\Omega)}=0 \enspace \mbox{for all} \enspace v_H\in V_H, \enspace v^\f \in \Vf.
\end{align}
The decomposition of $V_h$ in $\Vc$ and $\Vf$ is $a$-orthogonal
\begin{align}
\label{H1-orthogonality}V_h = \Vc \oplus \Vf \quad \mbox{and} \quad a(v^\cs,v^\f)=0 \enspace \mbox{for all} \enspace v^\cs\in \Vc, \enspace v^\f \in \Vf
\end{align}
and $L^2$-quasi-orthogonal in the sense that
\begin{align}
\label{L2-quasi-orthogonality}(v^\cs,v^\f)_{L^2(\Omega)} \lesssim H^2 \| \nabla v^\cs \|_{L^2(\Omega)} \| \nabla v^\f \|_{L^2(\Omega)}.
\end{align}
\end{lemma}
\begin{proof} 
The proof is verbatim the same as in \cite{Malqvist:Peterseim:2012}.
\end{proof}

The following lemma estimates the error of the best-approximation in the modified coarse space $\Vc$. The lemma is also implicitly required each time that we use the abstract error estimates stated in \cite[Theorem 1]{Cances:Chakir:Maday:2010}. These estimates require a family of finite-dimensional spaces that is dense in $H^1_0(\Omega)$. This density property is implied by the following lemma.

\begin{lemma}[Approximation property of $\Vc$]
\label{density}
For any given $v \in H^1_0(\Omega)$ with $\ddiv A\nabla v\in L^2(\Omega)$ it holds
\begin{align*}
\inf_{v^\cs_H \in \Vc} \|v - v^\cs_H \|_{H^1(\Omega)} \lesssim H\|\ddiv A\nabla v+bv\|_{L^2(\Omega)}+\inf_{v_h \in V_h} \|v - v_h \|_{H^1(\Omega)}.
\end{align*}
\end{lemma}
\begin{proof}
Given $v$, define $f_v:= \ddiv A \nabla v + b v \in L^2(\Omega)$ (since $v\in L^{\infty}(\Omega)$) and
let $v_h \in V_h$ denote the corresponding finite element approximation, i.e.,
\begin{align*}
a(v_h, \phi_h) = (f_v, \phi_h)_{L^2(\Omega)} \qquad \mbox{for all } \phi_h \in V_h.
\end{align*}
With $v_H^\cs:=\Pc v_h \in \Vc$, Galerkin-orthogonality leads to 
\begin{align*}
\|A^{1/2} \nabla (v_h - v_{H}^\cs)\|_{L^2(\Omega)}^2 &\overset{\eqref{H1-orthogonality}}{\leq} a( v_h, \Pf v_h )=(f_v,\Pf v_h)_{L^2(\Omega)} \\
&\overset{\eqref{e:errorclement}}{\lesssim} \gmin^{-1/2}  \|H f_v \|_{L^2(\Omega)} \|A^{1/2} \nabla(v_h - v_{H}^\cs) \|_{L^2(\Omega)}.
\end{align*}
This, the triangle inequality, and norm equivalences readily yield the assertion.\end{proof}

Next, we show that there exists an element $u^\cs= \Pc u_h$ in the space $\Vc$ that approximates $u_h$ in the energy-norm with an accuracy of order O$(H^2)$. 
\begin{lemma}[Stability and approximability of the reference solution]$\\$
\label{corollary-properties-eigenvalue-decompose}
Let $(u_h,\lambda_h)\in V_h\times\R$ solve Problem~\ref{discrete-problem-h}. Then it holds
\begin{align*}
\tnorm{\Pc u_h} &\le \sqrt{\lambda_h}, \\
\tnorm{\Pc u_h - u_h} = \tnorm{\Pf u_h} &\lesssim H^2 + H \| u- u_h \|_{H^1(\Omega)},  \\
( \Pc u_h,\Pf u_h)_{L^2(\Omega)} &\lesssim \left( H^2 + H \| u- u_h \|_{H^1(\Omega)} \right) H^2.
\end{align*}
\end{lemma}
\begin{proof}
Recall $\|\cdot\|_{H^1(\Omega)}:=\sqrt{a(\cdot,\cdot)}$. Since $\Pc$ is a projection, we have 
\begin{align*}
 \tnorm{\Pc u_h}^2\le \tnorm{u_h}^2 = \lambda_h \| u_h \|^2_{L^2(\Omega)} - \beta \| u_h \|^4_{L^4(\Omega)} \le \lambda_h.
\end{align*}
The $a$-orthogonality of \eqref{e:decompositionc} further yields
\begin{multline}
\label{bound-for-energy-uf}\tnorm{\Pf u_h}^2 = a( \Pf u_h,\Pf u_h) = a( u_h,\Pf u_h)\\
= \lambda_h ( u_h, (1-I_H)\Pf u_h )_{L^2(\Omega)} - \beta ( u^3 , \Pf u_h )_{L^2(\Omega)} - \beta ( u_h^3 - u^3 , \Pf u_h )_{L^2(\Omega)}.
\end{multline}
The first term on the right-hand side of \eqref{bound-for-energy-uf} can be bounded using $I_H(\Pf u_h) =0$, the $L^2$-orthogonality \eqref{L2-orthogonality}, and the estimates for the weighted Cl\'{e}ment interpolation operator \eqref{e:errorclement}
\begin{multline}\label{bound-for-energy-uf2}
\lambda_h ( u_h, (1-I_H)\Pf u_h )_{L^2(\Omega)}  = \lambda_h ( (1-I_H)u_h , (1-I_H)\Pf u_h )_{L^2(\Omega)}\\ \lesssim \lambda_h H^2 \tnorm{u_h} \tnorm{\Pf u_h}.
\end{multline}
Since $u\in L^{\infty}(\Omega)$ we have $\nabla (u^3) = 3 u^2 \nabla u \in L^2(\Omega)$ and, hence, the second term on the right hand side of \eqref{bound-for-energy-uf} can be bounded as follows, 
\begin{multline}
\beta ( u^3, \Pf u_h) = \beta ( (1-I_H)u^3, (1-I_H)\Pf u_h)
\overset{\eqref{local-estimate-clement}}{\lesssim} H^2 \| u \|^2_{L^{\infty}(\Omega)} \tnorm{u}\tnorm{\Pf u} \\\overset{\eqref{l-infty-estimate}}{\lesssim} H^2 \tnorm{u}\tnorm{\Pf u}.\label{bound-for-energy-uf3}
\end{multline}
Since $u_h^3-u^3=(u_h^2+u_h u+ u^2)(u_h-u)$, the third term on the right hand side of \eqref{bound-for-energy-uf} can be estimated by
\begin{multline}
\beta ( u_h^3 - u^3 , \Pf u_h )_{L^2(\Omega)} \lesssim \| |u| + |u_h| \|_{L^6(\Omega)}^2 \|u_h-u\|_{L^6(\Omega)} \| (1-I_H)\Pf u_h \|_{L^2(\Omega)} \\
\lesssim H \| u- u_h \|_{H^1(\Omega)}\tnorm{\Pf u_h},\label{bound-for-energy-uf4}
\end{multline}
where we used (\ref{u-energy}) and the embedding $\| |u| + |u_h| \|_{L^6(\Omega)} \lesssim \| u \|_{H^1(\Omega)} + \| u_h \|_{H^1(\Omega)}$.

The combination of \eqref{bound-for-energy-uf}--\eqref{bound-for-energy-uf4} readily yields
\begin{align*}
\tnorm{\Pf u_h} &\lesssim H^2 + \|u-u_h\|_{H^1(\Omega)}^2.
\end{align*}
The third assertion follows from the previous ones  and
\begin{multline*}
( \Pc u_h,\Pf u_h)_{L^2(\Omega)} = ((1-I_H)\Pc u_h,(1-I_H)\Pf u_h)_{L^2(\Omega)} \\\lesssim H^2 \tnorm{\Pc u_h} \tnorm{\Pf u_h}. 
\end{multline*}
\end{proof}

\subsection{Proof of Theorem~\ref{main-result-pre}}\label{proof-main-result-pre}

We split the proof into two parts: the estimate for the $H^1$-error and the estimate for the $L^2$-error.
\subsubsection{Proof of the $H^1$ error estimate \eqref{final-H1-estimate}}
We proceed similarly as in \cite{Cances:Chakir:Maday:2010}. The proof is divided into four steps. In the first step, we derive an identical formulation of some energy difference. The identity is used in step two to establish the inequality $\|u^{\cs}_h - u \|_{H^1(\Omega)}^2\lesssim E(u^{\cs}_h) - E(u)$. Since $u^{\cs}_h$ is a minimizer, we can replace $E(u^{\cs}_h)$ by $E(w^{\cs}_h)$ in the estimate for an arbitrary $L^2$-normalized $w^{\cs}_h\in \Vc$. In step three, we choose $w^{\cs}_h:=\frac{\Pc u_h}{\|\Pc u_h\|_{L^2(\Omega)}}$ and show that the perturbation introduced via normalization is of high order ($\approx H^3$). In step four, we use step three to estimate $E(w^{\cs}_h) - E(u)$.

$\\$
{\it Step 1.} Given some arbitrary $w \in H^1_0(\Omega)$ with $\|w\|_{L^2(\Omega)}=1$, we show that
\begin{equation}
\label{estimate-for-energies-proof-3} 
\begin{aligned}
 E(w) - E( u ) &= \frac{1}{2} a(w - u,w - u) + \frac{\beta}{2} (| u |^2 ( w - u ), w - u )_{L^2(\Omega)} \\
& \quad + \frac{\beta}{4} (( |u|^4 - 2 |u|^2 |w|^2 + |w|^4, 1 )_{L^2(\Omega)} - \frac{1}{2} \lambda \| w - u \|_{L^2(\Omega)}^2.
\end{aligned}
\end{equation}
First, using $\|u\|_{L^2(\Omega)}=\|w\|_{L^2(\Omega)}=1$ we get
\begin{align}
\nonumber\lambda ( u - w , u - w )_{L^2(\Omega)} &= \lambda \| u \|_{L^2(\Omega)}^2 - 2 \lambda (u,w)_{L^2(\Omega)} + \lambda \| u \|_{L^2(\Omega)}^2 \\
 \nonumber&= - 2 \lambda (u, w - u )_{L^2(\Omega)} \\
\label{proof-main-result-pre:step-1} &= - 2 a(u, w - u) - 2 \beta (|u|^2 u, w - u )_{L^2(\Omega)}.
\end{align}
This yields
\begin{eqnarray*}
\lefteqn{a(w,w) + \beta (|u|^2 w, w)_{L^2(\Omega)} - a(u,u) - \beta (|u|^2 u, u )_{L^2(\Omega)}}\\
&\overset{\eqref{proof-main-result-pre:step-1}}{=}& a(w,w) - 2 a(u, w) + a(u,u) \\
&\enspace& \quad + \beta (|u|^2  w, w)_{L^2(\Omega)} - 2 \beta (|u|^2 u, w)_{L^2(\Omega)} + \beta (|u|^2 u, u)_{L^2(\Omega)} \\
&\enspace& \quad - \lambda ( w - u , w - u )_{L^2(\Omega)}\\
&=& a(w - u,w - u) + \beta (|u|^2 (w - u), w - u )_{L^2(\Omega)} - \lambda \|w - u\|_{L^2(\Omega)}^2.
\end{eqnarray*}
Plugging this last equality into the equation
\begin{eqnarray*}
\nonumber\lefteqn{2 E(w) - 2 E( u )} \\
&=& a(w,w) + \frac{\beta}{2} (|w|^2 w,w )_{L^2(\Omega)} - a(u,u) - \frac{\beta}{2} (|u|^2 u,u )_{L^2(\Omega)}.
 \end{eqnarray*}
leads to \eqref{estimate-for-energies-proof-3}.

$\\$
{\it Step 2.}
Using \eqref{estimate-for-energies-proof-3} with $w=u^{\cs}_h$ and the fact that there exists some $c_0$ (independent of $H$ and $h$) such that $a(u-u^{\cs}_h,u-u^{\cs}_h)+((\beta|u|^2 - \lambda)(u-u^{\cs}_h), u-u^{\cs}_h )_{L^2(\Omega)} \ge c_0 \| u-u^{\cs}_h \|_{H^1(\Omega)}^2$ (c.f. \cite[Lemma 1]{Cances:Chakir:Maday:2010}), we get
\begin{eqnarray*}
\lefteqn{ E(u^{\cs}_h) - E( u )} \\
 &=& \frac{1}{2} a(u^{\cs}_h - u,u^{\cs}_h - u) + \frac{\beta}{2} (| u |^2 ( u^{\cs}_h - u ), u^{\cs}_h - u )_{L^2(\Omega)} \\
 &\enspace& \quad + \frac{\beta}{4} (( |u|^4 - 2 |u|^2 |u^{\cs}_h|^2 + |u^{\cs}_h|^4, 1 )_{L^2(\Omega)} - \frac{1}{2} \lambda \| u^{\cs}_h - u \|_{L^2(\Omega)}^2 \\
 &\ge& \frac{c_0}{2} \|u^{\cs}_h - u \|_{H^1(\Omega)}^2 + \frac{\beta}{4} \| |u|^2 - |u^{\cs}_h|^2 \|_{L^2(\Omega)}^2.
 \end{eqnarray*}
 {\it Step 3.}
Using the result of step two yields
\begin{align*}
\|u^{\cs}_h - u \|_{H^1(\Omega)}^2 \lesssim E(u^{\cs}_h) - E( u ) \le E(w^{\cs}_h) - E( u )
\end{align*}
for any $L^2$-normalized $w^{\cs}_h \in \Vc$. We choose $w^{\cs}_h:=\frac{\Pc u_h}{\|\Pc u_h\|_{L^2(\Omega)}}$ and observe that we get, with Lemma~\ref{corollary-properties-eigenvalue-decompose}, that
\begin{align}
\label{estimate-for-energies-proof-1}\nonumber\| \Pc u_h - w^{\cs}_h\|_{L^2(\Omega)} &=  \left|1 - \| \Pc u_h \|_{L^2(\Omega)} \right| \le \| \Pf u_h \|_{L^2(\Omega)} = \| \Pf u_h - I_H(\Pf u_h) \|_{L^2(\Omega)} \\
&\lesssim  H \| \Pf u_h \|_{H^1(\Omega)} \lesssim H \| u - u_h \|_{H^1(\Omega)}^2 + H^3
\end{align}
and consequently
\begin{align}
\label{estimate-for-energies-proof-2}\| \Pc u_h - w^{\cs}_h\|_{H^1(\Omega)} &= \frac{ \left|1 - \| \Pc u_h \|_{L^2(\Omega)} \right|}{\| \Pc u_h \|_{L^2(\Omega)}}  \| \Pc u_h \|_{H^1(\Omega)} \lesssim H \| u - u_h \|_{H^1(\Omega)}^2 + H^3,
\end{align}
where we used $\| u - u_h \|_{H^1(\Omega)} \lesssim 1$ (implying $\| \Pc u_h \|_{H^1(\Omega)} \lesssim 1$ and $\| \Pc u_h \|_{L^2(\Omega)} \gtrsim 1$).

$\\$
{\it Step 4.} Using again \eqref{estimate-for-energies-proof-3} leads to 
\begin{eqnarray*}
\lefteqn{2 E(w^{\cs}_h) - 2 E( u )} \\
&=& \|w^{\cs}_h - u\|_{H^1(\Omega)}^2 + \beta (| u |^2 ( w^{\cs}_h - u ), w^{\cs}_h - u )_{L^2(\Omega)} \\
&\enspace& \quad + \frac{\beta}{2} ( |u|^4 - 2 |u|^2 |w^{\cs}_h|^2 + |w^{\cs}_h|^4, 1 )_{L^2(\Omega)} -  \lambda \| w^{\cs}_h - u \|_{L^2(\Omega)}^2.
 \end{eqnarray*}
The H\"older-inequality
\begin{align}
\label{generalized-hoelder-for-step4}(|u|^2,|u - w^{\cs}_h|^2)_{L^2(\Omega)} \le \| u \|_{L^6(\Omega)}^2 \| u - w^{\cs}_h \|_{L^2(\Omega)} \| u - w^{\cs}_h \|_{L^6(\Omega)}
\end{align}
yields the estimate
\begin{eqnarray*}
 \lefteqn{\beta (| u |^2 ( w^{\cs}_h - u ), w^{\cs}_h - u )_{L^2(\Omega)} + \frac{\beta}{2} \int_{\Omega} \left( |u|^2 - |w^{\cs}_h|^2 \right)^2 \dx} \\
  &\overset{\eqref{generalized-hoelder-for-step4}}{\le}&\beta \|u\|_{L^6(\Omega)}^2 \| u - w^{\cs}_h \|_{L^2(\Omega)}  \| u - w^{\cs}_h \|_{L^6(\Omega)} + \frac{\beta}{2} ( ( |u| + |w^{\cs}_h| )^2, |u - w^{\cs}_h|^2 )_{L^2(\Omega)}\\
  &\overset{\eqref{generalized-hoelder-for-step4}}{\le}& \beta (2 \|u\|_{L^6(\Omega)}^2 + \|w^{\cs}_h\|_{L^6(\Omega)}^2 ) \| u - w^{\cs}_h \|_{L^2(\Omega)}  \| u - w^{\cs}_h \|_{L^6(\Omega)}\\
  &\lesssim& \| u - w^{\cs}_h \|_{L^2(\Omega)}^2 + \| u - w^{\cs}_h \|_{H^1(\Omega)}^2,
 \end{eqnarray*}
for the terms involving $\beta$. The combination of the previous results with Lemma~\ref{corollary-properties-eigenvalue-decompose} and estimates \eqref{estimate-for-energies-proof-1} and \eqref{estimate-for-energies-proof-2} gives us
\begin{align*}
\| u^{\cs}_h - u \|_{H^1(\Omega)}^2 \lesssim E(u^{\cs}_h) - E( u ) &\le E(w^{\cs}_h) - E( u ) \lesssim \| w^{\cs}_h - u \|^2_{H^1(\Omega)} \\
&\lesssim \| u - \Pc u_h \|^2_{H^1(\Omega)} + \|  \Pc u_h - w^{\cs}_h  \|^2_{H^1(\Omega)}\\
&\lesssim \left( \| u - u_h \|_{H^1(\Omega)} + H^2 \right)^2.
\end{align*}

\subsubsection{Proof of the $L^2$ error estimate \eqref{final-L2-estimate}}

In the following, we let the bilinear form $c_{\lambda,u}: H^1_0(\Omega) \times H^1_0(\Omega) \rightarrow \mathbb{R}$ be given by
\begin{align*}
c_{\lambda,u}(v,w):= \int_{\Omega} A \nabla v \cdot \nabla w + b v w + 3 \beta |u|^2 v w \dx - \lambda \int_{\Omega} v w \dx
\end{align*}
and we define the space
\begin{align*}
V_{u}^{\perp} := \{ v \in H^1_0(\Omega)| \hspace{2pt} (v,u)_{L^2(\Omega)}=0\}.
\end{align*}
For $w \in H^1_0(\Omega)$ we let $\psi_w \in \uorth$ denote the unique solution (see Lemma~\ref{cances-chakir-maday-theorem} below) of
\begin{align}
\label{adjoint-problem-for-L2}c_{\lambda,u}(\psi_{w},v_{\perp} ) = (w,v_{\perp})_{L^2(\Omega)} \quad \mbox{for all} \enspace v_{\perp} \in \uorth.
\end{align}
The subsequent lemma applies the abstract $L^2$-error estimate, obtained by Canc\`es, Chakir, Maday \cite[Lemma 1, Theorem 1, and Remark 2]{Cances:Chakir:Maday:2010}, to our setting. Observe that Lemma \ref{density} (i.e. $\Vc$ represents a dense family of finite dimensional subspaces of $H^1$) is required to apply these results.  
\begin{lemma}[Abstract approximation \cite{Cances:Chakir:Maday:2010}]
\label{cances-chakir-maday-theorem}
Let $h$ be sufficiently small, then
\begin{align}
\label{eigenvalue-estimate-ccm}|\lambda- \lambda_H^{\cs}| &\lesssim \| u - u_H^{\cs} \|_{H^1(\Omega)}^2  + \| u - u_H^{\cs} \|_{L^2(\Omega)}
\end{align}
and \begin{align}
\label{abstract-L2-estimate} \| u - u_H^{\cs} \|_{L^2(\Omega)}^2 &\lesssim \| u - u_H^{\cs} \|_{H^1(\Omega)}  \inf_{\psi \in \Vc} \| \psi_{u_H^{\cs} -u} - \psi \|_{H^1(\Omega)}.
\end{align}
Furthermore, the bilinear form $c_{\lambda,u}(\cdot,\cdot)$ is a scalar product in $H^1_0(\Omega)$ and induces a norm that is equivalent to the standard $H^1$-norm.
\end{lemma}
Observe the following equivalence. If $\psi_w \in \uorth$ solves
\begin{align*}
\int_{\Omega} A \nabla \psi_{w} \cdot \nabla v_{\perp} + b \psi_{w} v_{\perp} + \beta 3 |u|^2 \psi_w v_{\perp} \dx - \lambda \int_{\Omega} \psi_w v_{\perp} \dx = \int_{\Omega} w v_{\perp} \dx
\end{align*}
for all $v_{\perp} \in \uorth$, then it also solves
\begin{eqnarray*}
\lefteqn{\int_{\Omega} A \nabla \psi_{w} \cdot \nabla v + b \psi_{w} v + \beta 3 |u|^2 \psi_w  v \dx - \lambda \int_{\Omega} \psi_w v \dx}\\
&=& 2 \beta (u^3,\psi_w)_{L^2(\Omega)} \int_{\Omega} u v \dx + \int_{\Omega} (w - (w,u)_{L^2(\Omega)}) v \dx
\end{eqnarray*}
for all $v\in H^1_0(\Omega)$. This can be easily seen as follows: assume $\ddiv A \nabla \psi_w  \in L^2(\Omega)$ (the general result follows by density arguments) and let $P^{\perp} : L^2(\Omega) \rightarrow \uorth$ denote the $L^2$-orthogonal projection given by $P^{\perp}(v):=v-(v,u)_{L^2(\Omega)}$.
Since
\begin{align*}
\int_{\Omega} \left( -\ddiv A \nabla \psi_w  + b \psi_w + 3 \beta |u|^2 \psi_w - \lambda \psi_w \right) v^{\perp} \dx = \int_{\Omega} w v^{\perp} \dx.
\end{align*}
we get
\begin{align*}
\int_{\Omega} P^{\perp}\left( -\ddiv A \nabla \psi_w  + b \psi_w + 3 \beta |u|^2 \psi_w - \lambda \psi_w \right) v \dx = \int_{\Omega} P^{\perp}(w) v \dx
\end{align*}
for all $v \in H^1_0(\Omega)$. By using the explicit formula for $P^{\perp}$ and the definition of $u$ the reformulated equation follows. Furthermore, since $\psi_w \in H^1_0(\Omega)$ solves a standard elliptic problem, classical theory (c.f. \cite{Gilbarg}) applies and we get the $L^{\infty}$-estimate
\begin{align}
\label{L-infty-estimate-adjoint-problem}\|\psi_w \|_{L^{\infty}(\Omega)} &\lesssim (1+\lambda)\| \psi_w \|_{L^2(\Omega)} + |(|u|^3,\psi_w)| + \|w\|_{L^2(\Omega)} \lesssim 
(1+\lambda) \|w\|_{L^2(\Omega)}.
\end{align} 

\begin{lemma}[$L^2$-error estimate]
\label{L2-estimate-lemma}
Let $h$ be sufficiently small and let $u$ denote the solution of Problem~\ref{weak-problem}, $u_H^\cs$ the solution of Problem~\ref{discrete-problem-in-V-c}, and $\psi_{u-u_H^\cs} \in \uorth$ denote the solution of \eqref{adjoint-problem-for-L2} for $w=u-u_H^\cs$. Then
\begin{align*}
\|u - u_H^\cs \|_{L^2(\Omega)}
&\lesssim \left( \min_{\psi^h \in V_h} \frac{\|\psi_{u-u_H^\cs} - \psi^h \|_{H^1(\Omega)}}{\|u - u_H^\cs \|_{L^2(\Omega)}}  + H \right) \|u - u_H^\cs \|_{H^1(\Omega)}.
\end{align*}
\end{lemma}

In Lemma~\ref{L2-estimate-lemma}, the assumption that $h$ should be sufficiently small enters by using the $L^2$-estimate \eqref{abstract-L2-estimate}. Note that the coarse mesh size $H$ remains unconstrained. 

\begin{proof}
We define $e_H^\cs:=u - u_H^\cs$. Using Lemma~\ref{cances-chakir-maday-theorem} (and therefore implicitly Lemma~\ref{density}) we get
\begin{align}
\label{l2-estimate-in-proof}\frac{\|e_H^\cs \|_{L^2(\Omega)}^2}{\|e_H^\cs \|_{H^1(\Omega)}} \lesssim \|\psi_{u-u_H^\cs} - \psi^\cs_H \|_{H^1(\Omega)} \le \|\psi_{u-u_H^\cs} - \psi^h \|_{H^1(\Omega)} + \| \psi^\cs_H - \psi^h\|_{H^1(\Omega)}
\end{align}
for all $\psi^\cs_H \in \Vc$ and all $\psi^h \in V_h$. It remains to properly choose $\psi^h$ and $\psi^\cs_H$. The proof is structured as follows. We choose $\phi_h \in V_h$ to be the fine space approximation of the solution of the adjoint problem (\ref{adjoint-problem-for-L2}) and $\psi^\cs_H$ is chosen to be the $a(\cdot,\cdot)$-orthogonal approximation of $\psi^h$. This guarantees that $\psi^\cs_H - \psi^h$ is in the kernel of our interpolation operator (i.e.,  $I_H(\psi^\cs_H - \psi^h)=0$) and we can estimate the occurring terms while gaining an additional error order of $H$. The proof is detailed in the following.

Let us choose $\psi^h := \psi_{e_H^\cs}^h$, where $\psi_{e_H^\cs}^h\in V_h$ solves
\begin{eqnarray*}
c_{\lambda,u}(\psi_{e_H^\cs}^h,v_h) &=& 2 \beta (|u|^3,\psi_{e_H^\cs}^h)_{L^2(\Omega)} \int_{\Omega} u v_h \dx + \int_{\Omega} (e_H^\cs - (e_H^\cs,u)_{L^2(\Omega)}) v_h \dx
\end{eqnarray*}
for all $v_h \in V_h$. The coercivity of $c_{\lambda,u}$ and reinterpretation of the equation in the sense of problem \eqref{adjoint-problem-for-L2} yields that $\psi_{e_H^\cs}^h$ is well defined. Next, we define
\begin{align*}
g(v,w,u):= - \beta 3 |u|^2 v + \lambda v + 2 \beta (|u|^3,v)_{L^2(\Omega)} u + (w - (w,u)_{L^2(\Omega)})
\end{align*}
and solve for $\psi_{e_H^\cs}^{H,c} \in \Vc$ with
\begin{eqnarray*}
\int_{\Omega} A \nabla \psi_{e_H^\cs}^{H,c} \cdot \nabla v_H^\cs + b \psi_{e_H^\cs}^{H,c} v_H^\cs \dx &=& \int_{\Omega} g(\psi_{e_H^\cs}^h,e_H^\cs,u) v_H^\cs \dx
\end{eqnarray*}
for all $v_H^\cs \in \Vc$.
Since equally $\psi_{e_H^\cs}^h\in V_h$ fulfills
\begin{eqnarray*}
\int_{\Omega} A \nabla \psi_{e_H^\cs}^{h} \cdot \nabla v_h + b \psi_{e_H^\cs}^{h} v_h \dx &=& \int_{\Omega} g(\psi_{e_H^\cs}^h,e_H^\cs,u) v_h \dx
\end{eqnarray*}
for all $v_h \in V_h$, we obtain by using the $a(\cdot,\cdot)$-orthogonality of $\psi_{e_H^\cs}^{h}$ and $\psi_{e_H^\cs}^{H,c}$
\begin{align*}
a( \psi_{e_H^\cs}^{h} - \psi_{e_H^\cs}^{H,c},  \psi_{e_H^\cs}^{h} - \psi_{e_H^\cs}^{H,c} ) &= \int_{\Omega} g(\psi_{e_H^\cs}^h,e_H^\cs,u) (\psi_{e_H^\cs}^{h} - \psi_{e_H^\cs}^{H,c}) \dx\\
&\le \int_{\Omega} g(\psi_{e_H^\cs}^h,e_H^\cs,u) (\mbox{Id}-I_H)(\psi_{e_H^\cs}^{h} - \psi_{e_H^\cs}^{H,c}) \dx\\
&\lesssim ( \lambda \| \psi_{e_H^\cs}^h \|_{H^1(\Omega)} + \|e_H^\cs\|_{L^2(\Omega)} ) H \| \nabla (\psi_{e_H^\cs}^{h} - \psi_{e_H^\cs}^{H,c}) \|_{L^2(\Omega)}.
\end{align*}
Since
\begin{align*}
\| \psi_{e_H^\cs}^h \|^2_{H^1(\Omega)} \lesssim c_{\lambda,u}(\psi_{e_H^\cs}^h,\psi_{e_H^\cs}^h ) = (e_H^\cs,\psi_{e_H^\cs}^h)_{L^2(\Omega)},
\end{align*}
we get
\begin{align*}
\| \psi_{e_H^\cs}^h - \psi_{e_H^\cs}^{H,c} \|_{H^1(\Omega)} \lesssim H ( \| e_H^\cs \|_{L^2(\Omega)} + \lambda \| \psi_{e_H^\cs}^h \|_{H^1(\Omega)} ) \lesssim (1+\lambda) H \| e_H^\cs \|_{L^2(\Omega)}.
\end{align*}
Combining this estimate with \eqref{l2-estimate-in-proof} yields
\begin{align*}
\|u - u_H^\cs \|_{L^2(\Omega)} &\lesssim \left( \frac{\|\psi_{u-u_H^\cs} - \psi_{e_H^\cs}^{h} \|_{H^1(\Omega)}}{\|u - u_H^\cs \|_{L^2(\Omega)}}  + \frac{\|\psi_{e_H^\cs}^{h} - \psi_{e_H^\cs}^{H,c} \|_{H^1(\Omega)}}{\|u - u_H^\cs \|_{L^2(\Omega)}} \right) \|u - u_H^\cs \|_{H^1(\Omega)} \\
&\lesssim \left(\frac{\|\psi_{u-u_H^\cs} - \psi_{e_H^\cs}^{h} \|_{H^1(\Omega)}}{\|u - u_H^\cs \|_{L^2(\Omega)}}  + (1+\lambda)H \right) \|u - u_H^\cs \|_{H^1(\Omega)}\\
&\lesssim \left( \min_{\psi^h \in V_h} \frac{\|\psi_{u-u_H^\cs} - \psi^h \|_{H^1(\Omega)}}{\|u - u_H^\cs \|_{L^2(\Omega)}}  + H \right) \|u - u_H^\cs \|_{H^1(\Omega)}.
\end{align*}
In the last step we used C\'{e}a's lemma for linear elliptic problems and the fact that the $H^1$-best-approximation in the orthogonal space $\uorth \cap V_h$ can be bounded by the $H^1$-best-approximation in the full space $V_h$ (c.f. \cite{Cances:Chakir:Maday:2010} and equation (40) therein).
\end{proof}

Using \eqref{final-H1-estimate} and Lemma~\ref{L2-estimate-lemma} we obtain for $e_H^{\cs} := u - u_H^\cs$
\begin{align*}
\| e_H^{\cs} \|_{L^2(\Omega)}
&\lesssim \left( \min_{\psi^h \in V_h} \frac{\|\psi_{u-u_H^\cs} - \psi^h \|_{H^1(\Omega)}}{\|e_H^{\cs} \|_{L^2(\Omega)}}  + H \right) \| e_H^{\cs} \|_{H^1(\Omega)} \lesssim (|e_h^0| + H)\hspace{2pt}(|e_h^1| + H^2),
\end{align*}
where $|e_h^1| := \min_{v_h \in V_h}  \| u - v_h \|_{H^1(\Omega)}$ and $|e_h^0| := \min_{\psi^h \in V_h} \frac{\|\psi_{u-u_H^\cs} - \psi^h \|_{H^1(\Omega)}}{\|u - u_H^\cs \|_{L^2(\Omega)}}$.
Together with \eqref{eigenvalue-estimate-ccm} this yields
\begin{align*}
|\lambda- \lambda_H^\cs| &\lesssim \| e_H^{\cs} \|_{H^1(\Omega)}^2  + \| e_H^{\cs} \|_{L^2(\Omega)} \\
&\lesssim (|e_h^1|+ H^2)^2 + (|e_h^0| + H)\hspace{2pt}(|e_h^1| + H^2) \lesssim H |e_h^1| + H^3.
\end{align*}

\subsection{Proof of Theorem~\ref{main-result-post}}\label{proof-main-result-post}

Again, we split the proof into two subsections, one concerning $H^1$-error estimate and the other the $L^2$-error estimate.
\subsubsection{Proof of the $H^1$ error estimate \eqref{final-post-processed-H1-estimate}}
Due to the definitions of $u_h$ and $u_h^\cs$ we get for $v_h \in V_h$
\begin{eqnarray*}
\lefteqn{a(u_h-u^\cs_h, v_h)}\\
&=& \lambda_h(u_h,v_h) - \lambda^\cs_H( u^\cs_H , v_h ) - \beta ( |u_h|^2 u_h, v_h )_{L^2(\Omega)} + \beta ( |u^\cs_H|^2 u^\cs_H,v_h )_{L^2(\Omega)}\\
&=& \lambda_h(u_h-u^\cs_H,v_h) + (\lambda_h - \lambda^\cs_H)( u^\cs_H , v_h ) - \beta \sum_{i=0}^2 ( (u_h)^{2-i} (u^\cs_H)^i (u_h-u^\cs_H), v_h )_{L^2(\Omega)}.
\end{eqnarray*}
The treatment of the first and the second term in this error identity is obvious. The last term is treated with the H\"older-inequality and the embedding $H^1_0(\Omega) \hookrightarrow L^6(\Omega)$ (for $d\le3$):
\begin{eqnarray*}
\lefteqn{\sum_{i=0}^2 ( (u_h)^{2-i} (u^\cs_H)^i (u_h-u^\cs_H), v_h )_{L^2(\Omega)}} \\
&\le& \| u_h \|^2_{L^{6}(\Omega)} \| u_h-u^\cs_H \|_{L^2(\Omega)} \| v_h \|_{L^6(\Omega)} \\
&\enspace& \enspace + \| u_h \|_{L^{6}(\Omega)} \| u_H^\cs \|_{L^6(\Omega)} \| u_h-u^\cs_H \|_{L^2(\Omega)} \| v_h \|_{L^6(\Omega)} \\
&\enspace& \enspace +  \| u_H^\cs \|_{L^6(\Omega)}^2 \| u_h-u^\cs_H \|_{L^2(\Omega)} \| v_h \|_{L^6(\Omega)} \\
&\lesssim& \| u_h \|^2_{H^1(\Omega)} \| u_h-u^\cs_H \|_{L^2(\Omega)} \| v_h \|_{H^1(\Omega)} \\
&\enspace& \enspace + \| u_h \|_{H^1(\Omega)} \| u_H^\cs \|_{H^1(\Omega)} \hspace{2pt} \| u_h-u^\cs_H \|_{L^2(\Omega)}  \| v_h  \|_{H^1(\Omega)} \\
&\enspace& \enspace +  \| u_H^\cs  \|^2_{H^1(\Omega)} \| u_h-u^\cs_H \|_{L^2(\Omega)} \| v_h \|_{H^1(\Omega)}.
\end{eqnarray*}
We therefore get with $v_h = u_h-u^\cs_h$ and the Poincar\'{e}-Friedrichs inequality
\begin{align*}
\|u_h-u^\cs_h\|_{H^1(\Omega)} &\lesssim (\lambda_h+\lambda_H^\cs) \| u_h-u^\cs_H \|_{L^2(\Omega)} + |\lambda_h - \lambda^\cs_H|.
\end{align*}
This implies \eqref{final-post-processed-H1-estimate}.

\subsubsection{Proof of the $L^2$ error estimate in \eqref{final-post-processed-L2-estimate}}
\label{subsubsection-postprocessed-L2}

We start with a lemma that allows us to formulate an error identity.
\begin{lemma}
\label{lemma-L2-identity}Let $v \in H^1_0(\Omega)$ be an arbitrary function with $\|v\|_{L^2(\Omega)}=1$ and let $\psi_{u-v} \in \uorth$ denote the corresponding solution of the adjoint problem with
\begin{align*}
c_{\lambda,u}(\psi_{u-v},w_{\perp} ) = (u-v,w_{\perp})_{L^2(\Omega)}
\end{align*}
for all $w_{\perp} \in \uorth$ (c.f. \eqref{adjoint-problem-for-L2}). Then it holds
\begin{align*}
\|u-v\|_{L^2(\Omega)}^2 = c_{\lambda,u}(v-u,\psi_{u-v}) + \| u - v\|_{L^2(\Omega)}^2 \int_{\Omega} |u|^2 u \psi_{u-v} \dx + \frac{1}{4} \| u - v \|_{L^2(\Omega)}^4.
\end{align*}
\end{lemma}
The lemma can be extracted from the proofs given in \cite[pp. 99--100]{Cances:Chakir:Maday:2010}.

$\\$
The following lemma treats the semi-discrete case, i.e., we assume $V_h=H^1_0(\Omega)$. The reason is that the proof of the fully discrete case becomes very technical and hard to read. We note that the proof of the semi-discrete case analogously transfers to the fully-discrete case with sufficiently small $h$ by inserting additional continuous approximations to overcome the problems produced by the missing
uniform bounds for $\|u_h\|_{L^{\infty}(\Omega)}$ and $\|u_h^\cs\|_{L^{\infty}(\Omega)}$. For the readers convenience we therefore only prove the case $h=0$.
\begin{lemma}[Estimate \eqref{final-post-processed-L2-estimate} for $h=0$]
\label{lemma-post-processed-L2-bound}Assume $h=0$, i.e., $V_h=H^1_0(\Omega)$. Accordingly we let $u^\cs_0\in H^1_0(\Omega)$ denote the semi-discrete post-processed approximation, i.e., the solution to the problem
\begin{align*}
\int_{\Omega} A \nabla u^\cs_0 \cdot \nabla \phi \dx + \int_{\Omega} b u^\cs_0 \phi \dx = \lambda^\cs_H \int_{\Omega} u^\cs_H \phi \dx - \int_{\Omega} \beta |u^\cs_H|^2 u^\cs_H \phi \dx
\end{align*}
for all $\phi \in H^1_0(\Omega)$ (c.f. Problem~\ref{two-grid-postprocessing}). Then it holds
\begin{align*}
\|u -u^\cs_0\|_{L^2(\Omega)} &\lesssim H^4.
\end{align*}
\end{lemma}
 
\begin{proof}
We divide the proof into two steps. We want to make use of the error identity in Lemma \ref{lemma-L2-identity} with $v=u^\cs_0$. However, $u^\cs_0$ is not $L^2$-normalized and therefore no admissible test function in the error identity. In the first step, we therefore show that the normalization only produces an error of order $H^4$. In the second step it remains to show that the $L^2$-error between $u$ and the $L^2$-normalized $u^\cs_0$ is also of order $H^4$.

$\\$
{\it Step 1.} We show that $\left| || u^\cs_H ||_{L^2(\Omega)} - || u^\cs_0 ||_{L^2(\Omega)} \right|\lesssim H^4$, which implies $1 - H^4 \lesssim \|u_0^\cs\|_{L^{2}(\Omega)} \lesssim 1 + H^4$ (because of $|| u^\cs_H ||_{L^2(\Omega)}=1$).

First observe that $u^\cs_0\in H^1_0(\Omega)$ is the solution to a classical elliptic problem, which is why we obtain
\begin{align}
\label{L-infty-bound-u-c-0}\| u^\cs_0 \|_{L^{\infty}} \lesssim \lambda_H^\cs \lesssim \lambda.
\end{align}
Since $a(u^\cs_0- u^\cs_H, v^\cs_H) = 0$ for all $v^\cs_H \in \Vczero$ we get $u^\cs_0-u^\cs_H \in \Vfzero$. Hence
\begin{eqnarray*}
\lefteqn{a(u^\cs_0- u^\cs_H,u^\cs_0- u^\cs_H) = a(u^\cs_0, u^\cs_0- u^\cs_H)} \\
&=& \lambda_H^\cs (u_H^\cs,u^\cs_0- u^\cs_H) - \beta(|u_H^\cs|^2 u_H^\cs, u^\cs_0- u^\cs_H)\\
&=& \lambda_H^\cs (u_H^\cs,u^\cs_0- u^\cs_H) - \beta(|u_H^\cs|^2 u_H^\cs - |u|^2 u, u^\cs_0- u^\cs_H) - \beta(|u|^2 u, u^\cs_0- u^\cs_H).
\end{eqnarray*}
Using $u^\cs_0-u^\cs_H \in \Vfzero$ and inserting $I_H(u_H^\cs)$ and $I_H(u)$ several times, we get with similar arguments as above and with the previous estimate for $u_H^\cs-u$:
\begin{align*}
\| u^\cs_0- u^\cs_H \|_{H^1(\Omega)} \lesssim H^2
\end{align*}
and
\begin{align}
\label{u-c-0-u-c-H-L2-estimate}\| u^\cs_0- u^\cs_H \|_{L^2(\Omega)} = \| (u^\cs_0- u^\cs_H) - I_H(u^\cs_0- u^\cs_H) \|_{L^2(\Omega)} \lesssim H \| u^\cs_0- u^\cs_H \|_{H^1(\Omega)} \lesssim H^3.
\end{align}
Next, we show that $\left| \| u^\cs_0 \|_{L^2(\Omega)} -  1 \right|$ is of higher order. We start with:
\begin{eqnarray*}
\lefteqn{\| u^\cs_0 \|_{H^1(\Omega)}^2 - \|u^\cs_H \|_{H^1(\Omega)}^2 = a(u^\cs_0 ,u^\cs_0 ) - a (u^\cs_H,u^\cs_H) }\\
&=& \lambda_H^\cs (u_H^\cs, u_0^\cs - u_H^\cs)_{L^2(\Omega)} - \beta (|u_H^\cs|^2 u_H^\cs, (u^\cs_0 - u_H^\cs) )_{L^2(\Omega)} \\
&=& \lambda_H^\cs (u_H^\cs - I_H(u_H^\cs), u_0^\cs - u_H^\cs)_{L^2(\Omega)} \\
&\enspace& \quad - \beta (|u_H^\cs|^2 u_H^\cs - |u^\cs_0|^2 u_0^\cs , (u^\cs_0 - u_H^\cs) )_{L^2(\Omega)} - \beta ( |u^\cs_0|^2 u_0^\cs , (u^\cs_0 - u_H^\cs) )_{L^2(\Omega)}\\
&\overset{\eqref{u-c-0-u-c-H-L2-estimate}}{\lesssim}& (H^4 + H^6 - \beta ( |u^\cs_0|^2 u_0^\cs , u^\cs_0 - u_H^\cs )_{L^2(\Omega)}.
\end{eqnarray*}
Using that $u^\cs_0$ is bounded uniformly in $L^{\infty}(\Omega)$ we can proceed as in the proof of Lemma~\ref{corollary-properties-eigenvalue-decompose} to show:
\begin{align*}
 \beta ( |u^\cs_0|^2 u_0^\cs , u^\cs_0 - u_H^\cs )_{L^2(\Omega)} \lesssim H \|u^\cs_0\|_{H^1(\Omega)} \|u^\cs_0 - u_H^\cs \|_{L^2(\Omega)} \lesssim H^4.
\end{align*}
So in summary:
\begin{align*}
\left| \| u^\cs_0 \|_{H^1(\Omega)}^2 - \| u^\cs_H \|_{H^1(\Omega)}^2 \right| \lesssim H^4.
\end{align*}
However, on the other hand:
\begin{eqnarray*}
\lefteqn{\lambda_H^\cs \left( \| u^\cs_H \|_{L^2(\Omega)}^2 - \| u^\cs_0 \|_{L^2(\Omega)}^2 \right) }\\
&=& \lambda_H^\cs (u_H^\cs-u_0^\cs, u_0^\cs - I_H(u_0^\cs) )_{L^2(\Omega)} - \beta (|u_H^\cs|^2 u_H^\cs, (u^\cs_0 - u_H^\cs) )_{L^2(\Omega)} \\
&\enspace& \qquad - \| u^\cs_0 \|_{H^1(\Omega)}^2 + \| u^\cs_H \|_{H^1(\Omega)}^2.
\end{eqnarray*}
This we can treat with the previous results to get:
\begin{eqnarray*}
\left| || u^\cs_H ||_{L^2(\Omega)}^2 - || u^\cs_0 ||_{L^2(\Omega)}^2 \right| \lesssim H^4.
\end{eqnarray*}
With $|| u^\cs_H ||_{L^2(\Omega)}=1$ we get
\begin{align}
\label{difference-L2-norms-estimate}\left| || u^\cs_H ||_{L^2(\Omega)} - || u^\cs_0 ||_{L^2(\Omega)} \right| \le  \left| || u^\cs_H ||_{L^2(\Omega)}^2 - \| u^\cs_0 \|_{L^2(\Omega)}^2 \right| \lesssim H^4.
\end{align}
Note that in the last step we used that for any $a\ge 0$ it holds $|1-a|\le|1-a^2|$.

$\\$
{\it Step 2.} Step 1 justifies the definition of $\tilde{u}^\cs_0 := \| u^\cs_0 \|_{L^2(\Omega)}^{-1} u^\cs_0$ which fulfills
\begin{align}
\label{L2-estimate-bar-u-c-0-u-c-0}\| \tilde{u}^\cs_0 - u^\cs_0 \|_{L^2(\Omega)} =  \left|\| u^\cs_0 \|_{L^2(\Omega)} - 1 \right| \hspace{2pt} \| u^\cs_0 \|_{L^2(\Omega)} \lesssim H^4.
\end{align}
Next, we show $\|u-\tilde{u}^\cs_0 \|_{L^2(\Omega)} \lesssim H^4$. For this purpose define $\tilde{\lambda}_H^\cs:=\| u^\cs_0 \|_{L^2(\Omega)}^{-1}\lambda_H^\cs $. Then $\tilde{u}^\cs_0 \in H^1_0(\Omega)$ solves
 \begin{align*}
\int_{\Omega} A \nabla \tilde{u}^\cs_0 \cdot \nabla \phi \dx + \int_{\Omega} b \tilde{u}^\cs_0 \phi \dx = \tilde{\lambda}^\cs_H \int_{\Omega} u^\cs_H \phi \dx - \int_{\Omega} \frac{\beta}{\| u^\cs_0 \|_{L^2(\Omega)}} |u^\cs_H|^2 u^\cs_H \phi \dx.
\end{align*}
We want to use Lemma \ref{lemma-L2-identity} and denote $\psi:=\psi_{u-\tilde{u}^\cs_0}$ with $\psi_{u-\tilde{u}^\cs_0} \in V_u^{\perp}$ being the solution of \eqref{adjoint-problem-for-L2} for $w=u-\tilde{u}^\cs_0$. Before we start to estimate $c_{\lambda,u}(\tilde{u}^\cs_0-u,\psi)$ observe that $(u,\psi)_{L^2(\Omega)}=0$ (by definition) which yields
\begin{eqnarray}
\label{pp-estimate-L2-step2}\lefteqn{\lambda_H^\cs (u_H^{\cs}, \psi)_{L^2(\Omega)} - \lambda (\tilde{u}^\cs_0, \psi)_{L^2(\Omega)}}\\
\nonumber&=& (\lambda_H^\cs - \lambda ) (u_H^{\cs} - u, \psi)_{L^2(\Omega)} +  \lambda (u_H^{\cs} - u_0^{\cs}, \psi)_{L^2(\Omega)} + \lambda (u_0^{\cs} - \tilde{u}^\cs_0, \psi)_{L^2(\Omega)}.
\end{eqnarray}
We get:
\begin{eqnarray*}
\lefteqn{c_{\lambda,u}(\tilde{u}^\cs_0-u,\psi)}\\
&=&a(\tilde{u}^\cs_0-u,\psi) + 3 \beta \int_{\Omega} |u|^2 \tilde{u}^\cs_0 \psi \dx - 3 \beta \int_{\Omega} |u|^2 u \psi \dx - \lambda ( \tilde{u}^\cs_0, \psi )_{L^2(\Omega)} + \lambda ( u, \psi )_{L^2(\Omega)} \\
&=& a(\tilde{u}^\cs_0,\psi) + 3 \beta \int_{\Omega} |u|^2 \tilde{u}^\cs_0 \psi \dx - 2 \beta \int_{\Omega} |u|^2 u \psi \dx - \lambda ( \tilde{u}^\cs_0, \psi )_{L^2(\Omega)}\\
&=& \left( \frac{1-{\| u^\cs_0 \|_{L^2(\Omega)}}}{\| u^\cs_0 \|_{L^2(\Omega)}} + 1 \right) \left( \lambda^\cs_H \int_{\Omega} u^\cs_H \psi \dx - \beta \int_{\Omega} |u^\cs_H|^2 u^\cs_H \psi \dx \right) \\
&\enspace& \quad + 3 \beta \int_{\Omega} |u|^2 \tilde{u}^\cs_0 \psi \dx - 2 \beta \int_{\Omega} |u|^2 u \psi \dx - \lambda ( \tilde{u}^\cs_0, \psi )_{L^2(\Omega)}\\
&\overset{\eqref{pp-estimate-L2-step2}}{=}& \underset{=:\mbox{I}}{\underbrace{\left( \frac{1-{\| u^\cs_0 \|_{L^2(\Omega)}}}{\| u^\cs_0 \|_{L^2(\Omega)}} \right)  ( \lambda^\cs_H u^\cs_H - \beta |u^\cs_H|^2 u^\cs_H, \psi)_{L^2(\Omega)}}} 
+ \underset{=:\mbox{II}}{\underbrace{(\lambda_H^\cs - \lambda)( u_H^\cs - u, \psi )_{L^2(\Omega)}}}\\
&\enspace& \quad + \underset{=:\mbox{III}}{\underbrace{\lambda( u_H^\cs - u_0^\cs, \psi - I_H(\psi) )_{L^2(\Omega)}}} + \underset{=:\mbox{IV}}{\underbrace{\lambda( u_0^\cs - \tilde{u}_0^\cs, \psi )_{L^2(\Omega)}}}\\
&\enspace& \quad + \underset{=:\mbox{V}}{\underbrace{3 \beta (|u|^2(\tilde{u}_0^\cs -u_0^\cs),\psi )_{L^2(\Omega)}}} + \underset{=:\mbox{VI}}{\underbrace{3 \beta (|u|^2(u_0^\cs-u_H^\cs),\psi )_{L^2(\Omega)}}}\\
&\enspace& \quad - \underset{=:\mbox{VII}}{\underbrace{\beta \int_{\Omega} (u -u_H^\cs)^2(u_H^\cs + 2u)\psi \dx}}
\end{eqnarray*}
In the last step we used $(u,\psi)_{L^2(\Omega)}=0$ and
\begin{align*}
a^3 - 3 a b^2 + 2 b^3 = (a-b)^2(a+2b) \quad \mbox{for } a,b\in \mathbb{R}.
\end{align*}
With \eqref{difference-L2-norms-estimate} we have:
\begin{align*}
|\mbox{I}|\lesssim \left| \frac{1-{\| u^\cs_0 \|_{L^2(\Omega)}}}{\| u^\cs_0 \|_{L^2(\Omega)}} \right| ( \lambda^\cs_H +\| u^\cs_H \|^3_{H^1(\Omega)} ) \| \psi \|_{L^2(\Omega)} \lesssim H^4 \lambda^\cs_H (1+ (\lambda^\cs_H)^2) \| \psi \|_{H^1(\Omega)}.
\end{align*}
For II we use Theorem~\ref{main-result-pre} to obtain:
\begin{align*}
|\mbox{II}|\le |\lambda_H^\cs - \lambda| \| u_H^\cs - u \|_{L^2(\Omega)} \| \psi \|_{L^2(\Omega)} \lesssim H^3 H^3 \|\psi\|_{L^2(\Omega)} \le H^6 \|\psi\|_{H^1(\Omega)}.
\end{align*}
For term III we can use equation \eqref{u-c-0-u-c-H-L2-estimate} which gives us
\begin{align*}
|\mbox{III}|\le \lambda |( u_H^\cs - u_0^\cs, \psi - I_H(\psi) )_{L^2(\Omega)}| \lesssim \lambda H^3 \| \psi - I_H(\psi) \|_{L^2(\Omega)} \lesssim H^4 \|\psi \|_{H^1(\Omega)}.
\end{align*}
Using \eqref{L2-estimate-bar-u-c-0-u-c-0} we get
\begin{align*}
|\mbox{IV}|\le \lambda \| u_0^\cs - \tilde{u}_0^\cs\|_{L^2(\Omega)} \| \psi \|_{L^2(\Omega)} \lesssim H^4 \|\psi\|_{H^1(\Omega)}.
\end{align*}
Equally we get
\begin{align*}
|\mbox{V}|\lesssim |(|u|^2(\tilde{u}_0^\cs -u_0^\cs),\psi )_{L^2(\Omega)}| \lesssim \| u \|_{L^6(\Omega)}^2 \|\tilde{u}_0^\cs -u_0^\cs\|_{L^2(\Omega)} \|\psi \|_{L^6(\Omega)} \lesssim \lambda H^4 \|\psi\|_{H^1(\Omega)}.
\end{align*}
To estimate VI we need the $L^{\infty}$-estimate given by \eqref{L-infty-estimate-adjoint-problem} which reads 
\begin{align}
\label{L-infty-psi-used}\|\psi_{u-\tilde{u}_0^\cs}\|_{L^{\infty}(\Omega)} \lesssim \| \tilde{u}_0^\cs - u\|_{L^2(\Omega)}.
\end{align}
For $z \in \mathcal{N}_H$, let the values $u_z$ and $\psi_z$ denote the coefficients appearing in the weighted Cl\'ement interpolation of $u$ and $\psi$ (c.f. equation \eqref{def-weighted-clement}). Recall that $\Phi_z$ denote the nodal basis functions of $V_H$.
Using again \eqref{u-c-0-u-c-H-L2-estimate}, $(\Phi_z,u^\f)_{L^2(\Omega)}=0$ for all $z \in \mathcal{N}_H$, and the fact that $u_0^\cs-u_H^\cs\in \Vfzero$ we obtain
\begin{align*}
|\mbox{VI}|&\lesssim |(|u|^2(u_0^\cs-u_H^\cs),\psi )_{L^2(\Omega)}| \\
&= \left| ((u-I_H(u))u \psi, u_0^\cs-u_H^\cs)_{L^2(\Omega)} + \sum_{z\in\mathcal{N}_H} (u_z(u-u_z) \psi \Phi_z, u_0^\cs-u_H^\cs)_{L^2(\Omega)}  \right.\\
&\enspace \quad \left. + \sum_{z\in\mathcal{N}_H} (|u_z|^2(\psi-\psi_z) \Phi_z, u_0^\cs-u_H^\cs)_{L^2(\Omega)}
 \right|\\
 &\lesssim \| u \|_{L^{\infty}(\Omega)} \left( 2 \| \psi \|_{L^{\infty}(\Omega)}  \| u \|_{H^1(\Omega)} +  \| u \|_{L^{\infty}(\Omega)} \| \psi \|_{H^1(\Omega)} \right) H \|u_0^\cs-u_H^\cs\|_{L^2(\Omega)}\\
 &\overset{\eqref{L-infty-psi-used}}{\lesssim} H^4 \| \tilde{u}_0^\cs - u\|_{L^2(\Omega)}.
\end{align*}
For the last term Theorem~\ref{main-result-pre} leads to
\begin{align*}
|\mbox{VII}|\lesssim \|u - u_H^\cs\|^2_{H^1(\Omega)} \left( \|u_H^\cs\|_{L^2(\Omega)} + 2\|u\|_{L^2(\Omega)} \right) \|\psi\|_{H^1(\Omega)} \lesssim H^4 \|\psi\|_{H^1(\Omega)}.
\end{align*}
Combining the results for the terms I--VII and using $\|\psi\|_{H^1(\Omega)}\lesssim\|\tilde{u}^\cs_0-u\|_{L^2(\Omega)}$ we get
\begin{align*}
|c_{\lambda,u}(\tilde{u}^\cs_0-u,\psi)| \lesssim H^4 \|\tilde{u}^\cs_0-u\|_{L^2(\Omega)}.
\end{align*}
Since (by using the previous estimate for $\| u - \tilde{u}^\cs_0\|_{H^1(\Omega)}$)
\begin{align*}
\frac{1}{4} \| u - \tilde{u}^\cs_0 \|_{L^2(\Omega)}^4 + \| u - \tilde{u}^\cs_0\|_{L^2(\Omega)}^2 \int_{\Omega} |u|^2 u \psi_{u-\tilde{u}^\cs_0} \dx \le C H^3  \| u - \tilde{u}^\cs_0\|_{L^2(\Omega)}^2
\end{align*}
we finally obtain with Lemma~\ref{lemma-L2-identity}
\begin{eqnarray*}
\|u-\tilde{u}_0^\cs\|_{L^{2}(\Omega)}^2  \lesssim |c_{\lambda,u}(\tilde{u}^\cs_0-u,\psi)| \lesssim H^4 \|\tilde{u}^\cs_0-u\|_{L^2(\Omega)}.
\end{eqnarray*}
With \eqref{u-c-0-u-c-H-L2-estimate} we therefore proved
\begin{eqnarray*}
\|u- u_0^\cs\|_{L^{2}(\Omega)}  \lesssim H^4.
\end{eqnarray*}
\end{proof}

\begin{proposition}
The $L^2$-error estimate in the fully-discrete case can be proved analogously to the semi-discrete case above. We therefore get for sufficiently small $h$ that
\begin{align*}
\|u -u^\cs_h\|_{L^2(\Omega)} &\lesssim H^4 + C_{L^2}(h,H),
\end{align*}
with $C_{L^2}(h,H)$ behaving like the term $H^2 \| u - u_h \|_{H^1(\Omega)}$.
\end{proposition}

\subsubsection{Proof of the eigenvalue error estimate in \eqref{final-post-processed-L2-estimate}}
 
From the following corollary we can conclude estimate \eqref{final-post-processed-L2-estimate}.
\begin{corollary}
\label{corollary-post-processed-eigvl}Let $u^\cs_h \in V_h$ denote the solution of the post-processing step defined via Problem~\ref{two-grid-postprocessing} and let $\lambda^\cs_h:=(2 E(u^\cs_h) + 2^{-1} \beta \| u^\cs_h \|_{L^4(\Omega)}^4) \|u^\cs_h\|^{-2}_{L^2(\Omega)}$. Then there holds
\begin{align*}
|\lambda_h- \lambda^\cs_h | &\lesssim \| u_h - u_h^\cs \|^2_{H^1(\Omega)} + \| u_h - u_h^\cs \|_{L^2(\Omega)}. 
\end{align*}
\end{corollary}
\begin{proof}
We have for arbitrary $v_h \in V_h$:
\begin{eqnarray*}
\lefteqn{a(u_h-v_h,u_h-v_h) + \beta ( |u_h|^2( u_h-v_h ), u_h-v_h )_{L^2(\Omega)} - \lambda_h (u_h-v_h,u_h-v_h)_{L^2(\Omega)}} \\
&=& a(v_h,v_h) - \lambda_h(v_h,v_h) + \beta (|u_h|^2v_h,v_h)_{L^2(\Omega)}.\hspace{150pt}
\end{eqnarray*}
This implies with $v_h = u_h^\cs$
\begin{eqnarray*}
\lefteqn{|\lambda_h^\cs - \lambda_h |}\\
&=& \left| \frac{a(u_h^\cs,u_h^\cs) + \beta (|u_h^\cs|^2 u_h^\cs , u_h^\cs)_{L^2(\Omega)} - \lambda_h \|u_h^\cs\|^2_{L^2(\Omega)} }{\|u_h^\cs\|^2_{L^2(\Omega)}} \right| \\
&=& \left| \frac{ \| u_h-u_h^\cs \|^2_{H^1(\Omega)} + \beta ( |u_h|^2,( u_h-u_h^\cs )^2 )_{L^2(\Omega)} - \lambda_h \|u_h-u_h^\cs\|^2_{L^2(\Omega)}}{\|u_h^\cs\|^2_{L^2(\Omega)}} \right. \\
&\enspace& \qquad \quad \left. +  \frac{\beta ((|u_h|^2-|u_h^\cs|^2),|u_h^\cs|^2)_{L^2(\Omega)} }{\|u_h^\cs\|^2_{L^2(\Omega)}} \right|.
\end{eqnarray*}
The remaining estimate is straight forward using $(a^2-b^2)=(a-b)(a+b)$. Note that the last term is the dominating term.
\end{proof}

We obtain \eqref{final-post-processed-L2-estimate} from Corollary \ref{corollary-post-processed-eigvl} and our previous estimates for $\| u - u_h^\cs \|_{H^1(\Omega)}$ and $\| u - u_h^\cs \|_{L^2(\Omega)}$.

\end{document}